\documentclass[headings=small,parskip=half]{scrartcl}

\usepackage[utf8]{inputenc}
\usepackage[a4paper, top=1in, bottom=1in, left=0.9in, right=0.9in]{geometry} 
\usepackage{microtype} 

\usepackage{mathtools}
\usepackage{xparse}
\usepackage[numbers]{natbib}
\usepackage{authblk}
\usepackage{amsmath,amssymb}
\usepackage{lipsum}
\usepackage{dsfont}
\usepackage{subcaption}
\usepackage{enumerate}
\usepackage{hyperref} 
\usepackage{multicol}
\usepackage[dvipsnames]{xcolor}
\usepackage{MnSymbol}
\usepackage{amsthm}
\usepackage{bm}
\usepackage{float}
\usepackage{graphicx}
\graphicspath{ {./} }

\setcitestyle{square,numbers}
\newcommand{\E}{\mathbb{E}}

\numberwithin{equation}{section} 

\DeclarePairedDelimiterX{\Norm}[1]{\|}{\|}{\Normargs{#1}}
\NewDocumentCommand{\Normargs}{>{\SplitArgument{1}{;}}m}
{\Normargsaux#1}
\NewDocumentCommand{\Normargsaux}{mm}
{\IfNoValueTF{#2}{#1} {#1\nonscript\:\delimsize\vert\allowbreak\nonscript\:\mathopen{}#2}}%
\def\norm{\Norm*}%

\DeclarePairedDelimiterX{\set}[1]{\{}{\}}{\setargs{#1}}
\NewDocumentCommand{\setargs}{>{\SplitArgument{1}{;}}m}
{\setargsaux#1}
\NewDocumentCommand{\setargsaux}{mm}
{\IfNoValueTF{#2}{#1} {#1\nonscript\:\delimsize\vert\allowbreak\nonscript\:\mathopen{}#2}}%
\def\Set{\set*}%

\providecommand{\keywords}[1]{\textbf{Keywords. } #1}

\providecommand{\MSC}[1]{\textbf{MSC (2020). } #1}

\theoremstyle{plain}
\newtheorem{theorem}{Theorem}[section] 
\newtheorem{lemma}[theorem]{Lemma}
\newtheorem{proposition}[theorem]{Proposition}
\newtheorem{corollary}[theorem]{Corollary}

\theoremstyle{definition}
\newtheorem{definition}[theorem]{Definition} 
\newtheorem{example}[theorem]{Example}
\newtheorem{remark}[theorem]{Remark}
\newtheorem{assumption}{Assumption}

\renewcommand{\d}{\,\mathrm{d}}

\date{}

\begin{document}

\title{Stability of travelling wave solutions to reaction-diffusion equations driven by additive noise with Hölder continuous paths}
	
\author[1]{Amjad Saef}
\author[2]{Wilhelm Stannat}
\affil[1]{\small{
  Institute of Mathematics, Technische Universität Berlin,\linebreak
  Straße des 17.\ Juni 136, 10623 Berlin, Germany,\linebreak
  e-mail: \href{mailto:saef@math.tu-berlin.de}{saef@math.tu-berlin.de},\linebreak
  ORCID: \href{https://orcid.org/0009-0006-9503-3367}{0009-0006-9503-3367}}}
\affil[2]{\small{
  Institute of Mathematics, Technische Universität Berlin,\linebreak
  Straße des 17.\ Juni 136, 10623 Berlin, Germany,\linebreak
  e-mail: \href{mailto:stannat@math.tu-berlin.de}{stannat@math.tu-berlin.de},\linebreak
  ORCID: \href{https://orcid.org/0000-0002-0514-3874}{0000-0002-0514-3874}}}

\maketitle
\vspace{-1.7cm}

\begin{abstract}
In this paper we investigate stability of travelling wave solutions to a class of reaction-diffusion equations
perturbed by infinite-dimensional additive noise with Hölder continuous paths, covering in particular fractional Brownian motion with general Hurst index. We obtain long- and short time asymptotic error bounds on the maximal distance from the solution of the stochastic reaction-diffusion equation to the orbit of travelling wave fronts. These bounds, in terms of Hurst index and Hölder exponent, apply to a large class of infinite-dimensional self-similar drivers with Hölder continuous paths, such as linear fractional stable motion. We find that for short times, higher Hurst indices imply higher stability, while for large times, higher Hölder exponents imply higher stability.
\end{abstract}

\MSC{35K57, 60G22, 60H15, 92C20}

\section{Introduction}
In this work, we study the effect of additive noise with Hölder continuous paths on stability of travelling wave solutions to scalar reaction-diffusion equations of the form 
\begin{equation}\label{detNag}
\partial_t u(t,x) = Au(t,x) + f(u(t,x)),~ (t,x) \in [0,T] \times \mathcal O.
\end{equation}
We focus on the case where $A$ is some (usually partial differential) operator and $f$ is the superposition of an odd-order polynomial with negative leading-order coefficient and a globally Lipschitz function. Finally, $\mathcal O \subset \mathbb R^n$ is an open domain that is translation invariant in the direction of some unit vector $\nu$.
In this context, a travelling wave solution is a solution $v^{TW}$ of \eqref{detNag} such that $$v^{TW}(t) = \hat v(\cdot - ct \nu),~ t \geq 0$$ for some $\hat v \in C^2(\mathcal O)$ and $c>0$, the so-called wavespeed. It is known that if $\mathcal O$ is the real line or a cylindric domain, this equation admits a travelling wave solution if for example $A$ is a Laplacian and $f \colon \mathbb R \rightarrow \mathbb R$ is a cubic nonlinearity \cite{Hadeler1975TravellingFI,VolpertVolpertVolpert,Gardner1986ExistenceOM}. 

Of particular interest is the stability of travelling wave fronts under perturbations, as this property is a necessary feature of models that describe invading fronts that can be observed in experiments, for example spike propagation in nerve axons (cf. \cite{NagumoPulseTrans}). The aim of this work is to demonstrate path-dependent stability properties of travelling wave solutions to reaction-diffusion equations perturbed by infinite-dimensional noise with Hölder continuous paths. More specifically, we consider evolution equations of the form 
\begin{equation}\label{IntNag} \begin{cases} 
\d V(t) = (AV(t) + f\left(V(t)\right))\d t + \varepsilon \d N(t) \\  
V(0) = \hat v& 
\end{cases}\end{equation}
for operators $A$ and reaction terms $f$ as in equation \eqref{detNag} and some Hölder continuous path $N$ with values in a suitable function space. In particular, we investigate the effect of infinite-dimensional self-similar noise on the stability of travelling wave solutions. 

Our motivation to study this problem stems from the recent interest in scale-free and heavy tailed phenomena in neuroscience \cite{HeavyTailBrain, ScaleFeeBrain}. For example, it has been observed that power spectra of synaptic noise and local field potentials resemble that of temporally correlated noise with self-similar features \cite{Destexhe2022, DestexheHighConductance}. At the same time, activity patterns in neural populations \cite{BeggsAvalanches} and distributions of synaptic weights appear to exhibit heavy tails \cite{HeavyTailWeights}. In this light, it is not unreasonable to conjecture that even at the level of a single neuron, the dynamics are perturbed by random influences with self-similar or heavy-tailed distributions \cite{SelectiveNeuralAv}. Mathematical models of such dynamics were studied for example in \cite{FitzHughNagNonGaussian, ChaosHeavyTailed,NeuralNetsLevyNoise}, while experimental evidence of heavy-tailed noise in synaptic clefts has been explored in \cite{NonGaussianNeurotransmitters}. We aim to contribute to the investigation of the presence of such phenomena by isolating effects of different Hurst indices and temporal fluctuations on mesoscopic dynamics.

From a mathematical perspective, self-similar noise processes are of intrinsic interest as stochastic objects, due to their rich structural properties. This applies of course to the extensively studied (fractional) Brownian motion, but also includes non-Gaussian processes, see for example the monographs \cite{TaqquSamorodnitsky} and \cite{TudorBook}. Since Lévy motions and Hermite processes combine central limit properties with stationary increments and self-similarity, they suggest themselves as natural models for non-Brownian noise and have been used as drivers of stochastic differential equations. Fractional Brownian motion in particular is a widely used model of persistent or anti-persistent behaviour in time series, and has been applied in hydrology, telecommunications, and turbulence theory \cite{Mandelbrot2001GaussianSA}. Similarly, fractional $\alpha$-stable Lévy processes capture heavy-tailed displacements with long-range dependence, making them suitable for modelling anomalous diffusion and intermittent transport phenomena \cite{IntermittentStuff}.

Stochastic (partial) differential equations driven by processes like fractional Brownian motion (fBm) require specialised techniques for integration, as the classical Itô calculus framework may no longer apply. Theories extending the probabilistic approach of Itô calculus, based for example on the Skorokhod integral (cf. \cite{AlosNualart}) or Russo-Vallois type integrals (cf. \cite{RussoVallois}), have been successfully adapted to general Gaussian and other non-standard noises. Alternatively, more analytically flavoured approaches which extend Young's pathwise integration theory have proven to be particularly successful in recent decades. Notably, rough path theory (cf. \cite{Lyons1998}) has been shown to provide a robust framework for handling differential equations driven by irregular signals, including those with low regularity, by encoding higher-order information about the signals' paths. In the scope of this paper, due to the smoothing properties of the semigroups generated by the operators $A$ that we consider, we can apply a simpler theory of Young integration against Hölder continuous paths developed in \cite{Gubinelli2004YoungIA} to obtain a solution concept that is appropriate in our setting. A strength of this approach is that it generalises the $L^2$-theory of integration applied e.g. in \cite{Pasik-Duncan2006} and yields pathwise bounds.

In the deterministic setting, the orbit $\Gamma = \Set{\hat v(\cdot - t \nu); t \in \mathbb R}$ of the translates of the travelling wave profile has been shown to be stable for diverse instances of equations which fit the framework of equation \eqref{detNag}. Here, we say that $\Gamma$ is stable in a suitable Banach space $\mathcal B$ if for any $\delta > 0$, $$\sup_{t \geq 0} \mathrm{d}(u(t),\Gamma) \coloneqq \sup_{t \geq 0} \inf_{\phi \in \mathbb R} \norm{u(t)-\hat v(\cdot - \phi \nu)}_{\mathcal B} \leq \delta$$ for any solution $u$ of equation \eqref{detNag} with $\norm{u_0 - \hat v}$ small enough. These results were primarily obtained by analysis of the spectral properties of the linearisation of \eqref{detNag} around translates of the travelling wave front (see for example \cite{EvansStabilityNerve, henry81:GTS, KapitulaMultDimStab}) or maximum principles and comparison theorems, as pioneered in \citet{Fife1977TheAO}. While the latter methods 
cannot be transferred to the stochastic setting without
unnatural monotonicity conditions on the noise term, the first approach can generally be applied to the stochastic case. However, these qualitative perturbation results on the spectrum of the linearisation generally yield limited quantitive information. For this reason, we follow e.g. \cite{KruegerStannat} and \cite{LangStannat} and make use of functional inequalities of Poincaré type to gain tighter estimates on the decay of perturbations orthogonal to the tangential direction of $\Gamma$.

Demonstrating stability in stochastic settings additionally requires adapted definitions of what constitutes a travelling wave front and how to verify its stability under successive perturbations. In particular, the phase of the stochastically perturbed wave front is not necessarily uniquely specified. In the last decade, several articles on stability of travelling waves under stochastic forcing tackled this problem by introducing (stochastic) ordinary differential equations whose solutions approximate the phase of the travelling wave front $\tilde v \in \Gamma$ with minimal distance (in e.g. $\mathcal B = L^2(\mathbb R)$ or $\mathcal B = H^1(\mathbb R)$) to the observed stochastic process (cf. \cite{StannatTW,KruegerStannat, InglisMacLaurin, HamsterHupkes,EichingerKuehn}). Subsequently, these works showed that the distance of the stochastic travelling wave to the specific wavefront remains small. 

In this work, we follow the method first implemented in \cite{StannatTW}, that consists of introducing 
a time-dependent phase correction $C(t)$ following gradient descent dynamics minimizing the $L^2$-distance between the solution $V$ to equation \eqref{IntNag} and $\Gamma$. Let $\tilde v^{TW}(t) \coloneqq \hat v\left(\cdot - C(t)\nu\right)$ and $\tilde U(t) \coloneqq V(t) - \tilde v^{TW}(t)$. As in \cite{KruegerStannat} and \cite{EichingerKuehn}, this phase correction is designed so that $\tilde U(t)$ can be decomposed as 
\begin{equation} 
\label{decomposition} 
\tilde U(t) = \varepsilon Z_\varepsilon(t) + y_\varepsilon(t),
\end{equation} 
where $Z_\varepsilon$ denotes an Ornstein-Uhlenbeck type process which satisfies 
\begin{equation} \label{NewOU}
\d Z_\varepsilon(t) = A(t) Z_\varepsilon(t) \d t + \varepsilon \d N(t)
\end{equation} for a family of dissipative operators $(A(t))_{t \in [0,T]}$, and $y_\varepsilon(t)$ denotes the residual nonlinear part. An analysis of this decomposition then yields that $\sup_{0 \leq t \leq T}\norm{\tilde U}_{L^2(\mathcal O)}$ remains small for small noise amplitudes $\varepsilon > 0$. Evidently, it then follows that $\mathrm{d}\left(V(t), \Gamma\right)$ remains small for all $t \in [0,T]$.

In order to solve equation \eqref{NewOU} for general Hölder continuous drivers $N$, we introduce a simple extension of the framework introduced in \cite{Gubinelli2004YoungIA}. This approach employs pathwise convolutions against evolution systems generated by linear perturbations of injective, sectorial operators. Pathwise mild approaches to non-autonomous evolution equations perturbed by an irregular path have previously been developed for (possibly multiplicative) Wiener noise in \cite{PronkVeraar}, and applied to prove existence of random attractors of dynamical systems generated by SPDEs with additive noise and quasi-linear drift in \cite{Kuehn2020RandomAV}.

To reach the desired stability result, we exploit dissipativity of the operator $A(t)$ and show that the squared norm of the residual $y_\varepsilon$ satisfies a differential inequality which implies a bound of the form $$ \sup_{0 \leq t\leq T}\norm{y_\varepsilon(t)}_{L^2} \in o\left(\varepsilon \sup_{0 \leq t\leq T}\norm{Z_\varepsilon(t)}_{L^2 \cap L^{r+1}}\right)$$ for initial conditions $\tilde U(0) = y_\varepsilon(0) = 0$  and $\varepsilon$ small enough. Here $r$ denotes the degree of the odd-order polynomial nonlinearity. Thereby, we find that \begin{equation} \label{oestimate}
\sup_{0 \leq t \leq T} \mathrm{d}(V(t),\Gamma) \leq \sup_{0 \leq t \leq T} \norm{\tilde U(t)}_{L^2} \leq \varepsilon \sup_{0 \leq t \leq T}\norm{Z_\varepsilon(t)}_{L^2} + o\left(\varepsilon \sup_{0 \leq t\leq T} \norm{Z_\varepsilon(t)}_{L^2 \cap L^{r+1}}\right)
\end{equation} for any small enough $\varepsilon > 0$. To make the dependence on the path of the driver $N$ explicit, we derive $\varepsilon$-independent bounds on $\norm{Z_\varepsilon}_{L^\infty([0,T];L^2\cap L^{r+1})}$, in terms of the time $T$ and the Hölder norm on $[0,T]$ of the driver $N$. To reach these estimates, we  extend an integration by parts formula for stochastic convolutions against fractional Brownian motion (cf. \cite{SchmalfussMaslowski}) to general Hölder continuous paths. Applying these bounds to inequality \eqref{oestimate}, we obtain the main results of the pathwise stability analysis, Proposition \ref{ShortLongBounds}.

The main results of this article concern $H$-self-similar drivers $X^H$ with Hölder regularity $\eta<H$. Theorems \ref{mainthm}, \ref{LongboundSelfSim} and Corollary \ref{AppliedSelfSim}, propose estimates for the probability that the error $d(V(t), \Gamma)$ satisfies the short time bound \begin{equation} \label{IntroShortBound}
\sup_{0 \leq t \leq T} d(V(t), \Gamma)\lesssim T^{H} + o\left(T^{H} \right), ~ T \ll 1
\end{equation} and, given $0<\delta$, the long-time bound \begin{equation} \label{IntroLongBound}
\sup_{0 \leq t \leq T} d(V(t), \Gamma)\lesssim  \delta + o\left(\delta\right),~ T \gg 1.
\end{equation} for $\varepsilon \sim T^{H-\eta}$. We achieve this by combining the pathwise bounds on $y_\varepsilon(t)$ with tail estimates on Hölder norms of self-similar processes with values in Banach spaces. 

The short term bound \eqref{IntroShortBound} shows a direct relationship between the Hurst index $H$ of the driver and stability on short times. As $H$ can in principle be arbitrarily large and $\eta$ arbitrarily low, we see that the time-regularity of the driver is secondary to these estimates. However, for large times, the degree of Hölder continuity determines the necessary scaling. Let \begin{equation}
\eta_X \coloneqq \sup\{\eta<H  :X^H \text{ is Hölder continuous with exponent } \eta\}.
\end{equation} The large time bound \eqref{IntroLongBound} shows that for for any $\eta<\eta_X$, the scaling $\varepsilon \sim T^{-(H-\eta)}$ yields stability on large time scales (cf. Thm \ref{LongboundSelfSim}). It is notable that in important examples (cf. Appendix \ref{TailEstimateSection}), the degree of Hölder continuity is largely determined by tails of increments of the driver.

To illustrate the scope of applications, we include an overview of relevant Hurst parameters and Hölder exponents in the appendix. Therein, we derive tail estimates for important explicit examples such as  fractional $\alpha$-stable Lévy motion, which includes fractional Brownian motion. Though these results are direct generalisations of finite-dimensional arguments, such tail estimates on Hölder norms of infinite-dimensional fractional stable Lévy motions have to our knowledge not yet been derived.

The article is organised as follows: In Section \ref{Prelims}, we provide an overview of the mathematical setting and assumptions we work with, including examples of reaction-diffusion equations that satisfy the stated assumptions. Then, in Section \ref{Existence}, we introduce the relevant notions of solutions to equations perturbed by noise with Hölder continuous paths and prove pathwise existence and uniqueness in the setting we specified in the previous section. Section \ref{pathstab} derives decomposition \eqref{decomposition} and provides estimates on the residual $y_\varepsilon (t)$ of the type 
$$
\norm{y_\varepsilon}_{L^\infty([0,T];L^2)}\in o\left(\varepsilon \norm{Z_\varepsilon}_{L^\infty\left([0,T];L^2 \cap L^{r+1}\right)}\right). 
$$
This effectively shows that $Z_\varepsilon$ is a first-order approximation of $\tilde U$. We finish Section \ref{pathstab} by deriving explicit bounds on $Z_\varepsilon$ in terms of the Hölder norm of the driver $N$ and the time $T$. Section \ref{SectionfBM} then combines the pathwise results with tail estimates for Hölder norms of Banach space-valued self-similar noise to derive long- and short term asymptotics on the distance of perturbed travelling waves to the orbit of travelling wave fronts.

\section{Preliminaries} \label{Prelims}
\subsection{Setting} \label{Setting}
In this section, we present the fundamental assumptions and notation relevant to this article. Generally, the constant $C$ stands for a generic constant that changes possibly from line to line.

Let $\mathcal O$ be some open domain $\mathcal O \subset \mathbb R^n$, $n \geq 1$, that is invariant under translation in the direction of some vector $\nu \in \mathbb R^n$. To cover the wide range of functions that (fractional powers of) differential operators can be applied to, we at first consider a linear operator $A$ defined on some subspace $D(A)$ of the vector space $L^0(\mathcal O)$ of Borel-measurable functions up to almost sure equivalence.

\begin{assumption} \label{TranslInv}
Our first assumption on $A$ is that the operator commutes with translation, i.e. $f \in D(A)$ iff $f(\cdot + c\nu) \in D(A)$ given any $c \in \mathbb R$, and $$Af(\cdot+c\nu) = (Af)(\cdot+c\nu).$$ 
\end{assumption}
We place the following further assumptions on the operator $A$.
\begin{assumption} \label{AnaSem}
To ensure that $A$ generates an analytic semigroup when restricted to $L^p$-spaces, we assume that for all $1 < p < \infty$ the operator $$-A_p \coloneqq -A|_{D_p(A)} \colon D_p(A) \subset L^p(\mathcal O) \rightarrow L^p(\mathcal O)$$ is injective and sectorial on $L^p(\mathcal O)$, where $$D_p(A) \coloneqq \Set{f \in L^p(\mathcal O) \cap D(A); Af \in L^p(\mathcal O)}$$ is assumed to be dense in $L^p(\mathcal O)$.  
\end{assumption}
Note that in particular, this implies that the spectrum of $-A_p$ does not intersect with $(-\infty,0)$. Additionally, we obtain a scale of Banach spaces $(\mathcal B^p_\gamma)_{\gamma \in \mathbb R}$ defined as the domains of fractional powers $$\mathcal B^p_\gamma = D\left((-A_p)^\gamma\right)$$ for $\gamma \geq 0$, while for $\gamma < 0$, we define $$\mathcal B^p_\gamma = \overline{L^p(\mathcal O)}^{\norm{(-A)^\gamma \cdot}_{L^p(\mathcal O)}},$$ i.e. the completion of $L^p(\mathcal O)$ under the norm $\norm{(-A)^\gamma \cdot}_{L^p(\mathcal O)}$.
\begin{assumption} \label{SlfAdj}
As an operator $$A_2 \colon D_2(A) \subset L^2(\mathcal O) \rightarrow L^2(\mathcal O),$$ the generator $A$ is self-adjoint. Equipped with the inner product $$\langle f,g \rangle_{\gamma} \coloneqq \langle f, g \rangle_{L^2} + \langle (-A)^\gamma f, (-A)^{\gamma}g \rangle_{L^2}$$ induced by $(-A)^\gamma$, we find that the spaces $\mathcal B^2_\gamma$ are in fact Hilbert spaces for $\gamma \geq 0$.
\end{assumption}

\begin{assumption} \label{NemytskiiAss}
Now, let a Nemytskii operator $f \colon \mathbb R \rightarrow \mathbb R$ of the form $f= f_0 + f_1$ be given, where $f_0$ is a polynomial of odd degree $r \geq 3$ with negative leading order coefficient, and $f_1 \in C^2(\mathbb R)$ with bounded first and second derivative.
\end{assumption}

\begin{assumption} \label{IntPl}
In order to establish the stability properties in Section \ref{pathstab}, we need a type of Gagliardo-Nirenberg interpolation inequality given by \begin{equation} \label{IntPlIneq}\norm{u}_{L^p(\mathcal O)} \leq C_p \norm{(-A_2)^{1/2} u}^{\theta_p}_{L^2(\mathcal O)}\norm{u}^{1-\theta_p}_{L^2(\mathcal O)}\end{equation} for $u \in \mathcal B^2_{1/2} = D((-A_2)^{1/2})$ and $3 \leq p \leq (r+1)$. Here $r$ denotes the growth exponent of the polynomial nonlinearity $f_0$. The constants $C_p, \theta_p$ depend on the exponent $p$. We additionally assume that \begin{equation} \label{ExponentCondition}
\theta_p < 2/p
\end{equation} holds for $3 \leq p \leq r$, but not necessarily for $p = r+1$.
\end{assumption}
\begin{remark}
Under the same conditions on $C_p, \vartheta_p$, our results still hold if we replace \eqref{IntPlIneq} by the weaker version $\norm{u}_{L^p} \leq C_p \norm{u}^{\theta_p}_{\mathcal B^2_{1/2}}\norm{u}^{1-\theta_p}_{L^2(\mathcal O)}.$ The affected Lemmas \ref{diffineq} and \ref{Zbound} then still follow after minor modifications of the involved constants.
\end{remark}

\begin{remark}
The main results of this work are also applicable for $f_0 \equiv 0$ or $f_0 \equiv -cx$ for some $c > 0$.  In this case, inequality \eqref{IntPlIneq} only needs to hold for $p = 3$, with $\theta_3 < 2/3$. The existence theorem remains nearly unchanged (cf. Remark \ref{linearf0}) and the changes in subsequent theorems are laid out in Remark \ref{f0equal0}.
\end{remark}

\begin{example} \label{FracLap}
Let $\Delta \colon H^2(\mathbb R^d) \subset L^2(\mathbb R^d) \rightarrow L^2(\mathbb R^d)$ denote the Laplace operator on $\mathcal O = \mathbb R^d$. Then, for all $s > 0$, $A = -(-\Delta)^s$ satisfies Assumptions \ref{TranslInv} to \ref{SlfAdj} (cf. \cite{Komatsu1966FractionalPO}, Thm 10.5). We now specify conditions for Assumption \ref{IntPl} to hold. For $p = q = 2$, the fractional Gagliardo-Nirenberg inequality \cite{fractionalGalNir} states that for $2 \leq p < \infty$ and $s \geq \left(\frac12 - \frac{1}{p}\right)d > 0$, it holds that \begin{equation} \label{GalNir}
\norm{u}_{L^p} \leq C_p \norm{(-\Delta)^{s/2}u}^{\theta_p}_{L^2} \norm{u}^{1-\theta_p}_{L^2}
\end{equation}
with $\theta_p = \left(\frac12 - \frac{1}{p}\right) \frac{d}{s}.$ As an example, for $d = r = 3$, one would need $s \geq \frac{3}{4}$ to obtain the interpolation inequality \eqref{IntPlIneq} in the case $p = r+1 = 4$. If $s$ is strictly larger than $\frac{3}{4}$, then condition \ref{ExponentCondition} is also fulfilled.
\end{example}

\begin{example} \label{CylDom}
Consider $\mathcal O = \mathbb R \times \mathbb T$, where $\mathbb T$ denotes the flat torus. Let $A = \partial_{xx} -(-\partial_{yy})^s$, for simplicity with $0 < s < 1$. Then, it is known that Assumptions \ref{TranslInv} to \ref{SlfAdj} hold in this setting.

We can further verify Assumption \ref{IntPl}. We first verify the inequality for the fully fractional Laplacian $- (-\partial_{xx}-\partial_{yy})^s$ on the specified domain. We embed $\mathcal O$ into $\tilde{\mathcal O} = \mathbb R \times [0,1]$, which is a Sobolev extension domain (see Theorem 8.6 in \cite{FracSobolevFirstCourse}), so any function $W^{s,2}(\mathcal O)$ embeds continuously into $W^{s,2}(\mathbb R^2)$ with bounded scaling of the $L^2$-norm. The usual Gagliardo-Nirenberg inequality then applies. For $u \in W^{s,2}(\tilde{\mathcal O})$, let $\tilde u$ denote its extension. Further, let $\dot W^{s,2}$ denote the homogeneous Sobolev-Slobodeckij space. Then, for $r,s, C_p$ and $\theta_p$ chosen as in the previous example,$$
\begin{aligned}
\norm{u}_{L^r(\tilde{\mathcal O})} \leq  \norm{\tilde u}_{L^r(\mathbb R^2)} &\leq C_p \norm{(-\Delta)^{s/2}\tilde u}^{\theta_p}_{L^2(\mathbb R^{2})} \norm{\tilde u}^{1-\theta_p}_{L^2(\mathbb R^{2})} \\ & \leq C\cdot C_p \norm{\tilde u}^{\theta_p}_{\dot W^{s,2}(\mathbb R^2)} \norm{ \tilde u}^{1-\theta_p}_{L^2(\mathbb R^2)} \\ &\leq  C\cdot C_p \norm{ u}^{\theta_p}_{\dot W^{s,2}(\tilde{\mathcal O})} \norm{  u}^{1-\theta_p}_{L^2(\tilde{\mathcal O})},
\end{aligned}$$ where the second line follows by equivalence of the homogeneous Sobolev-Slobodeckij space with the corresponding Bessel potential space \cite{fractionalGalNir}. Now, if we include periodic boundary conditions, then it actually holds that $\norm{ u}_{\dot W^{s,2}(\mathcal O)}$ is equivalent to $\norm{(-\Delta)^{s/2} u}^{\theta_p}_{L^2(\mathcal O)}$, defined in the semi-discrete frequency space. In conclusion, we demonstrated the desired inequality. A quick calculation in frequency space then shows that the operator $A$ inherits this bound and we can conclude that Assumption \ref{IntPl} holds. 

We note that this reasoning extends to higher-dimensional cylindrical domains, and, in the case $s = 1$, to the Dirichlet and Neumann Laplacian, where Assumption \ref{AnaSem} follows by maximal $L^p$-regularity \cite{Nau2011} and \ref{IntPlIneq} is verified as above.
\end{example}

\subsection{Travelling wave solutions} \label{TWSection}
The central object of study of this article is a travelling wave profile $v^{TW}_0.$
To simplify notation, define $$v^{TW}_{t} \coloneqq v^{TW}_0(\cdot - t \nu), ~ t \in \mathbb R$$ where $\nu \in \mathbb R^n$ denotes the direction of wave propagation. For simplicity, we assume that the vector $\nu$ is of unit length in the Euclidean norm on $\mathbb R^n$.
\begin{assumption}\label{TW} We assume that $v_0^{TW} \in D(A) \cap C^1(\mathcal O) \cap L^\infty(\mathcal O).$ In particular, $v^{TW}_{ct} \in D(A)$ for all $t \in \mathbb R$, as a consequence of translation invariance of $A$. We further assume that the travelling wave front $t \mapsto v^{TW}_{ct}$ solves equation \eqref{detNag}, i.e. \begin{equation} \label{WavEq}
\partial_t v^{TW}_{ct} = - c \nu \cdot \nabla v^{TW}_{ct} = Av_{ct}^{TW} + f(v^{TW}_{ct})
\end{equation}
for all $t \in \mathbb R$ and some $c \in \mathbb R$, henceforth called the \textit{wave speed}.  
\end{assumption}
\begin{remark}
By translation invariance, \eqref{WavEq} only needs to hold for $c=0$ for Assumption \ref{TW} to be satisfied. Further, it automatically follows that $t \mapsto v^{TW}_{ct+k}$ is a travelling wave solution for any $k \in \mathbb R$.
\end{remark}
\begin{assumption} \label{Intgrbl}
We assume that the travelling wave front's directional derivative in the direction of propagation is square integrable, i.e. $$\nu \cdot \nabla v^{TW}_0 \in L^2(\mathcal O).$$ Additionally, if $c = 0$ so that we have a standing wave solution, we assume that $v^{TW}_0$ is twice differentiable and $$\nu \cdot H(v^{TW}_0) \cdot \nu \in L^2(\mathcal O),$$ where $H(v^{TW}_0) = (\partial_{ij}v^{TW}_0)_{1 \leq i,j \leq n}$ denotes the Hessian matrix.
\end{assumption}
\begin{remark} \label{TechnicalRemark}
The purpose of Assumption \ref{Intgrbl} is mainly technical. In the scope of this paper, we avoid treating $A$ as an operator on, for example, the space of bounded continuous functions. The intended effect is that the space in which equation \eqref{WavEq} lives remains unspecified. However, $L^2$-integrability of the spatial derivative combined with boundedness of $v^{TW}_c$ implies that $$A(v^{TW}_{ct_1}-v^{TW}_{ct_2}) = -c\nu\cdot \nabla(v^{TW}_{ct_1}-v^{TW}_{ct_2}) - (f(v^{TW}_{ct_1})-f(v^{TW}_{ct_2})) \in L^2(\mathcal O).$$
It follows that $v^{TW}_{ct_1}-v^{TW}_{ct_2} \in D_2(A) \subset L^2(\mathcal O)$, and we can instead work with the properties of $A$ as a sectorial operator on $L^2(\mathcal O)$.
\end{remark}

The following assumption will be crucial to our stability analysis. This type of inequality ensures exponential decay of perturbations orthogonal to the tangential direction of the manifold. Thus, we can verify stability of the manifold of travelling wave fronts. In this case, the function $\nu \cdot \nabla v^{TW}_{x}$ is heuristically assumed to be a tangential vector to a point $v^{TW}_{x} \in \Gamma$. This assumption can be supported for example by noting that $$v^{TW}_{ct}-v^{TW}_{cs} = -c \int_s^t \nu \cdot \nabla v^{TW}_{cr} \d r,$$ so we can see that perturbations in the direction of $\nu \cdot \nabla v^{TW}_{x}$ do not generally decay. 
\begin{assumption} \label{SpcGp}
We assume that the $L^2(\mathcal O)$-linearisation $\mathcal L_{TW}u \coloneqq Au + f'(v^{TW}_{c_0})u$ around any translate $v^{TW}_{c_0}$ of the travelling wave front exhibits a spectral gap inequality of the form \begin{equation} \label{SpcGpIneq}
\langle \mathcal L_{TW} u, u \rangle_{L^2} \leq - \kappa_\ast \norm{u}_{\mathcal B^2_{1/2}}^2 + C_\ast \langle \nu \cdot \nabla v^{TW}_{c_0}, u \rangle^2_{L^2}, 
\end{equation} for some $k_\ast, C_\ast > 0$ independent of the phase $c_0$ and any $u \in \mathcal B^2_{1/2}$.
\end{assumption}

We give now give examples of domains $\mathcal O$ and operators $A$ together with reaction potentials $f$ such that travelling wave solutions exist and Assumptions \ref{TranslInv} to \ref{SpcGp} are satisfied.

\begin{example} \label{StannatBistable}
Consider the deterministic bistable reaction-diffusion equation $$\partial_t v(t,x) = \nu \partial_{xx} v(t,x) + f(v(t,x)),~ v(0,x) = v_0(x)$$ for $(t,x) \in \mathbb R_+ \times \mathbb R$. As $A = \nu \Delta = \nu \partial_{xx}$ is the usual Laplace operator on $\mathbb R$, Assumptions \ref{TranslInv}, \ref{AnaSem}, \ref{SlfAdj} and \ref{IntPl} are satisfied. In particular, Assumption \ref{IntPl} holds for $2 \leq r \leq 5$. For such equations, travelling wave solutions are assured to exist if $f \colon \mathbb R \rightarrow \mathbb R$ is a continuously differentiable function satisfying 
$$\begin{aligned}
&f(0) = f(a) = f(1) = 0 ~ ~ ~ \text{for some } a  \in (0,1) \\
&f(x) < 0 ~~~ \text{for } x \in (0,a), f(x) > 0 ~~~ \text{for } x \in (a,1) \\
& f'(x) < 0, f'(a) > 0, f'(1) < 0
\end{aligned}$$
Under these conditions, one can verify \cite{Hadeler1975TravellingFI} existence of a monotone increasing travelling wave front $\hat v$ connecting the stable fixed points $0$ and $1$ of the reaction term. It then actually holds \cite{StannatStabilityBistable} that $\partial_x \hat v, \partial_{xx} \hat v \in L^2(\mathbb R)$, so that Assumption \ref{Intgrbl} is satisfied. Let $a \in (0,1)$ be the unique zero of $f$ in the interval $(0,1)$. Under the additional assumption that there exists some $v_\ast \in (a,1)$ such that $f''(v) > 0$ on $(0,v_\ast)$ and $f''(v) < 0$ on $(v_\ast,1)$, it was shown in \cite{StannatStabilityBistable} that Assumption \ref{SpcGp} is fulfilled. Note that we do not assume any growth conditions on the nonlinearity $f$.
\end{example}
\begin{example}
The previous example generalises to cylindrical domains $\mathcal O = \mathbb R \times \mathbb T$, where $\mathbb T$ denotes the flat torus. For simplicity, assume that its volume $\mu(\mathbb T) = 1$ is normalised. Let $a, b > 0$ be real numbers and let $\hat v$ denote the wave front from Example \ref{StannatBistable}. Define $v^{TW}(t,x,y) \coloneqq \hat v(x-ct)$, so that this front is constant in the $y$-direction. Let $A = -\alpha (-\Delta_y)^{s}$ denote the fractional Laplacian on the periodic domain, defined by its action in frequency space. Thus $A v^{TW}(t) = 0$. It follows immediately that $$\begin{aligned}
\partial_t v^{TW}(t) = -c \partial_x \hat v(\cdot-ct) &= a \partial_{xx} \hat v(\cdot-ct) + bf(\hat v(\cdot-ct))= a\partial_{xx}v^{TW}(t)+Av^{TW}(t) + b f(v^{TW}(t)).
\end{aligned}$$  Phrased differently, we see that this actually defines a travelling wave solution. 

Further, we have seen in Example \ref{CylDom} that this operator satisfies Assumptions \ref{TranslInv} to \ref{SlfAdj}, and Assumption \ref{IntPl} for $r = 3$ and $s > \frac12$. Given strong enough coercive properties of $A$, the spectral gap property \ref{SpcGp} is inherited:
$$\begin{aligned}
\langle \nu \partial_{xx}u + A u+f'(v^{TW})u,u\rangle_{L^2} &\leq - \alpha \norm{(-\Delta_y)^{s/2} u}^2_{L^2} + \int_{\Omega} - a \int_{\mathbb R} (\partial_x u)^2 + f'(v^{TW})u^2\d x \d y
\\ &\leq - \alpha \norm{(-\Delta_y)^{s/2}u}^2_{L^2}- \kappa_\ast \norm{\partial_x u}^2_{L^2} + C_\ast \int_\Omega\left(\int_{\mathbb R} \partial_x \hat v \cdot u\d x\right)^2 \d y.
\end{aligned}$$
This is almost in the desired form; the only necessary ingredient to arrive at the spectral gap inequality is the fractional Poincaré inequality $$\norm{f-\int_\mathbb T f \d y}_{L^2} \leq \frac{1}{(2\pi)^s} \norm{(-\Delta)^{s/2}f}_{L^2}$$ which follows for $s > 0$ and $f \in D((-A)^\frac12)$ since $$\sum_{k \in \mathbb Z \setminus \{0\}} |\hat f(k)|^2 = \sum_{k \in \mathbb Z \setminus \{0\}} \frac{(2\pi k)^s}{(2\pi k)^s}  |\hat f(k)|^2 \leq \frac{1}{(2\pi)^s} \sum_{k \in \mathbb Z} (2\pi k)^s|\hat f(k)|^2.$$ By the Fubini identity, $(-\Delta)^{s/2}\left(\int_{\mathbb R} \partial_x \hat v(x) \cdot u(x,\cdot)\d x \right)(y)= \int_{\mathbb R} \partial_x \hat v(x) (-\Delta)^{s/2}\left(u(x,\cdot)\right)(y)\d x$ and thus
$$\int_\Omega\left(\int_{\mathbb R} \partial_x \hat v \cdot u\d x\right)^2 \d y \leq (2\pi)^{-s} \int_\mathcal{O} \partial_x \hat v |(-\Delta)^{s/2}u|^2 \d x \d y + C_\ast \left(\int_{\mathcal O} \partial_x \hat v \cdot u \d x \d y\right)^2,$$ where we could apply Jensen's inequality since $\partial_x \hat v \in L^1(\mathbb R)$. Now assume $\alpha > C_\ast C\norm{\partial_x \hat v}_{L^\infty}$.

We remark that this example generalises to $\mathbb T^n$, $n > 1$, and also applies if $A$ is a symmetric, constant coefficient strongly uniformly elliptic operator in $\Omega \subset \mathbb R^d$, with Neumann boundary conditions. 
\end{example}

\begin{remark}
The investigation of the spectral gap property \ref{SpcGp} could be an interesting direction of future research. Consider $A = -\gamma\partial_{xxxx} + \alpha \partial_{xx}$ or $A = -(-\Delta)^s$ for $1/2 < s < 1$ on the real line $\mathbb R$, and reaction potentials $f = f_0 + f_1$ as in Example \ref{StannatBistable}. Existence of travelling waves $\hat v^{\gamma}, \hat v^s$ which satisfy Assumptions \ref{TW} and \ref{Intgrbl} was shown in \cite{JanBergTWFourthOrder} and \cite{TWfracLap, CHAN20174567}, and one can further show that these terms satisfy Assumptions \ref{TranslInv} to \ref{IntPl} for suitable parameter ranges. 
\end{remark}

\section{Existence of Solutions} \label{Existence}
Before we can define and demonstrate our notion of stability, we need to specify the concept of solution of the evolution equation
\begin{equation}\label{Nag} \begin{cases} 
\d V(t) = (AV(t) + f\left(V(t)\right))\d t + \varepsilon \d N(t) \\  
V(0) = V_0 & 
\end{cases}\end{equation}
perturbed by some Hölder continuous path $N$. To simplify notation, let $v^{TW}(t) \coloneqq v^{TW}_{ct}$ denote the travelling wave solution of equation \eqref{detNag} with initial condition $v^{TW}(0) = v^{TW}_0.$ To circumvent technical difficulties, we formally decompose $$V(t) = v^{TW}(t) + U(t) = v^{TW}(t) + (V(t)-v^{TW}(t))$$ and instead solve the semilinear equation 
\begin{equation}\label{diffNag} \begin{cases} 
\d U(t) = (AU(t) + f(U(t)+v^{TW}(t)) - f(v^{TW}(t)))\d t + \varepsilon \d N(t) \\  
U(0) = U_0 & 
\end{cases}\end{equation}
on $[0,T] \times L^2(\mathcal O)$ where $U_0 = V_0 - v^{TW}_0$. Let $(S(t))_{t \geq 0}$ denote the semigroup generated by the operator $A$. Then any prospective mild solution of equation \eqref{diffNag} should satisfy the identity $$U(t) = S(t)U_0 + \int_0^t S(t-s)\left(f(U(s)+v^{TW}(s)) - f(v^{TW}(s))\right)\d s + \varepsilon \int_0^t S(t-s) \d N(s)$$ for all $t \in [0,T]$. For this to be well-defined, we first need to define the process $$N_A(t) \coloneqq  \int_0^t S(t-s) \d N(s)$$ for a given Hölder continuous path $N$. To this end,  we follow \citet{Gubinelli2004YoungIA} to introduce a pathwise notion of convolution of an analytic semigroup against a Hölder continuous path. For a given Banach space $\mathcal B$, define $$C^\eta_0([0,T];\mathcal B) = \left\{ f \text{ is Hölder continuous with } f(0) = 0\right\}.$$ Equipped with the Hölder seminorm $\norm{\cdot}_{C^\eta_0}$ on $\mathcal B$, $C^\eta_0$ is in fact a Banach space. Let $L$ be some sectorial operator on a Banach space $\mathcal B$ and $X \in C^\eta([0,T];\mathcal B)$.  Naturally, for initial condition equal to $0$, it should hold that a mild solution of the Cauchy problem $$\mathrm d Y(t) = L Y(t) \mathrm d t + \mathrm d X(t)$$ satisfies $$Y(t) \coloneqq \int_0^t P(t-s) \d X(s),$$ where $(P(t))_{t \geq 0}$ denotes the analytic semigroup generated by $L$. We define this process by means of a Riemann-Stieltjes integral against $N$. 
\begin{remark}
The notation $C^\eta_0$ should not be confused with the space of $\eta$-Hölder continuous functions $f$ such that $$\lim_{\delta \to 0} \sup_{\substack{ 0\leq s < t \leq 1 \\ |t-s| < \delta}} \frac{\norm{f(t)-f(s)}_\mathcal{B}}{|t-s|^\eta} = 0.$$
\end{remark}
\begin{definition} Let $\mathcal B$ be a Banach space and $(P(t))_{t \geq 0}$ be a strongly continuous semigroup on $\mathcal B$. For a function $X \colon [0,t] \rightarrow \mathcal B$, we define the convolution of $X$ against the semigroup  generated by $A$ as the limit \begin{equation}\label{riemsum}
\int_0^t P(t-s) \d X(s) \coloneqq \lim_{n \to \infty} \sum_{t^n_k \in \pi_n, t^n_k < t} P(t-t^n_k)(X(t^n_{k+1})-X(t^n_k)),
\end{equation} whenever this limit exists uniquely for any sequence of partitions $\pi_n = \Set{t^n_0, \dots , t^n_{k_n}}$ of $[0,t]$ such that $|\pi_n| \to 0$.
\end{definition}
\begin{proposition}[\cite{Gubinelli2004YoungIA}] \label{YoungConv}
Suppose that $L \colon D(L) \subset \mathcal B \rightarrow \mathcal B$ is an injective sectorial operator. Let $(\mathcal B_\gamma)_{\gamma \in \mathbb R}$ be the induced scale of Banach spaces (cf. Assumption \ref{AnaSem}) and $(P(t))_{t \geq 0}$ be the analytic semigroup generated by $L$. Let $X \in C^\eta([0,T];\mathcal B_{-\gamma})$ with $\gamma < \eta$. Then, given $\delta \in [0,\eta - \gamma)$ and $\kappa \in (0, \min(1,\eta - \gamma - \delta))$, the limit \eqref{riemsum} exists in $\mathcal B_\delta$ and in particular, $$Y(\cdot) \coloneqq \int_0^\cdot P(\cdot - s) \d X(s) \in C^\kappa_0([0,T];\mathcal B_\delta).$$ Further, for all $T > 0$, there exist constants $C_1$ and $C_2$ only dependent on $\eta, \gamma, \delta$ and $\kappa$ such that
$$\norm{Y}_{C^\kappa_0([0,T];\mathcal B_\delta)} \leq C_1 \norm{X}_{C^\eta_0([0,T];\mathcal B_{-\gamma})}$$ and $$\sup_{t \in [0,T]} \norm{Y(t)}_{\mathcal B_\delta} \leq C_2 T^{\eta - \gamma - \delta} \norm{X}_{C^\eta_0([0,T];\mathcal B_{-\gamma})}.$$
\end{proposition}
\begin{remark}
Let $r \geq 1$. Tracing through the proof of Proposition \ref{YoungConv}, we observe that if $$X \in C^\eta([0,T];\mathcal B^2_{-\gamma} \cap \mathcal B^{r+1}_{-\gamma}),$$ then $$Y \in C^\kappa_0([0,T];\mathcal B^2_{\delta} \cap \mathcal B^{r+1}_{\delta}) \subset C^\kappa_0([0,T];L^2(\mathcal O) \cap L^{r+1}(\mathcal O))$$ for $\delta, \kappa$ as in the Proposition \ref{YoungConv}. Note that $L^2(\mathcal O) \cap L^{r+1}(\mathcal O)$ canonically embeds into $\mathcal B^2_{-\gamma} \cap \mathcal B^{r+1}_{-\gamma}$, since this latter space is defined as the completion of $L^2(\mathcal O) \cap L^{r+1}(\mathcal O)$ under the norm $$\norm{(-A)^{-\gamma} \cdot }_{L^2(\mathcal O) \cap L^{r+1}(\mathcal O)} \coloneqq \norm{(-A)^{-\gamma} \cdot }_{L^2(\mathcal O)}  \vee \norm{(-A)^{-\gamma} \cdot }_{L^{r+1}(\mathcal O)},$$ and this completion is generally strictly larger.
\end{remark}
\subsection{Stochastic Nagumo equation driven by Hölder noise}
In what follows, we omit the dependence of the properties of the operator $A$ on the domains $D_p(A)$ (cf. Assumption \ref{AnaSem}) when it is evident from the context. Further, we set $L^p \coloneqq L^p(\mathcal O)$ and let $\langle \cdot , \cdot \rangle$ be the dual pairing on $L^p$, and for ease of notation, define $\mathcal U \coloneqq L^2 \cap L^{r+1}$ and  $\mathcal U_{-\gamma} \coloneqq \mathcal B^2_{-\gamma} \cap \mathcal B^{r+1}_{-\gamma}$. 

\begin{definition} \label{SolDef}
Let $(S(t))_{t \geq 0}$ denote the strongly continuous semigroup generated by the operator $A$ and let $N$ be a Hölder continuous path in $C^\eta_0([0,T];\mathcal U_{-\gamma})$ for $0 \leq \gamma < \eta$. A process $V$ with $$V - v^{TW} \in C([0,T];L^2) \cap L^{r+1}([0,T];L^{r+1})$$ is a solution of equation \eqref{Nag} if $$V(t)-v^{TW}(t) = S(t)\left(V_0-v^{TW}_0\right) + \int_0^t S(t-s)\left(f(V(s))-f(v^{TW}(s))\right) \d s + \varepsilon \int_0^t S(t-s) \d N(s)$$ for all $t \in [0,T]$.
\end{definition}
\begin{remark}
This last identity is indeed well-defined as the conditions on $f$ ensure that $f(V(s))-f(v^{TW}(t))$ can be decomposed into a sum of a number of elements of $L^p$-spaces with differing exponents $1 < p < \infty$, so that application of the semigroup is well-defined.
\end{remark}
\begin{theorem} \label{exuniq}
Let $v_0 \in v^{TW}_0 + L^2$ and $N \in C^\eta_0([0,T];\mathcal U_{-\gamma})$ for $0 \leq \gamma < \eta$ be given. Suppose that the operator $A$ satisfies Assumptions \ref{TranslInv}, \ref{AnaSem} and \ref{SlfAdj} and let the Nemytskii operator $f = f_0 + f_1$ satisfy Assumption \ref{NemytskiiAss}. Then there exists a mild solution to equation \eqref{diffNag}. This solution satisfies $$V-v^{TW} \in C([0,T];L^2) \cap L^{r+1}([0,T];L^{r+1}) \cap L^2([0,T];\mathcal B^2_\delta)$$ for any $0 \leq \delta < \eta - \gamma$ and is the unique mild solution $V$ such that $$V-v^{TW}\in C([0,T];L^2) \cap L^{r+1}([0,T];L^{r+1}),$$ where $r$ denotes the degree of the polynomial nonlinearity $f_0$.
\end{theorem}
\begin{remark} \label{linearf0}
If $f_0 \equiv 0$ or if $f_0$ is linear, then the problem reduces to the case of a Lipschitz nonlinearity. The above theorem then still holds, now with the trivial choice $r = 1$. In particular,  $L^{r+1}([0,T];L^{r+1}) = L^2([0,T];L^2)$, $\mathcal U = L^2$ and $\mathcal U_{-\gamma} = \mathcal B^2_{-\gamma}$.
\end{remark}

The proof of Theorem \ref{exuniq} proceeds in several steps. Let $N_A \in C^\kappa_0([0,T];\mathcal U_\delta)$ denote the convolution obtained from Proposition \ref{YoungConv}. By subtracting the convolution $N_A$ from candidate solutions of equation \eqref{Nag}, we reduce this problem to a partial differential equation. In order to solve \eqref{Nag}, we show that there exists a unique mild and variational solution $$w \in C([0,T];L^2) \cap L^2([0,T];\mathcal B^2_{1/2}) \cap L^{r+1}([0,T];L^{r+1})$$ to the partial differential equation \begin{equation}\label{varNag}
\partial_t w(t) = A w(t) + f\left(w(t) + N_A(t) + v^{TW}(t)\right) - f\left(v^{TW}(t)\right)
\end{equation}
The proof of existence via a Faedo-Galerkin approximation follows similar arguments as the proof sketched in Ch. 3 of \cite{Temam1997} (Thm. 1.1). For the sake of completeness, we include an overview of the proof in our particular case.

We aim to verify the conditions of Theorem 5.1.3 in \citet{LiuRoeckner2015} in the deterministic case to prove existence of variational solutions with respect to the Gelfand triple $$ \mathcal B^2_{1/2} \hookrightarrow L^2 \hookrightarrow ( \mathcal B^2_{1/2})^\ast.$$ For $X \colon [0,T] \times \mathcal O \rightarrow \mathbb R$ define $$F_X(t,u) \coloneqq f\left(u + X(t) + v^{TW}(t) \right) - f\left(v^{TW}(t)\right),~ t \in [0,T], ~u \colon \mathbb R \rightarrow \mathbb R$$ 
We note the following useful properties of the operator $F$, which can be verified by direct calculation.
\begin{lemma} \label{moncoer}
There exist some generic constants $K, C > 0$ such that for any $X \colon [0,T] \rightarrow \mathcal U$ and arbitrary $t \in [0,T]$, $u,v \in \mathcal B_{1/2}$,
\begin{enumerate}
    \item[(1)] $\langle F_X(t,u) - F_X(t,v),u-v\rangle \leq \mathrm{Lip}_f \norm{u-v}_{L^2},$ \label{test}
    \item[(2)] $\langle F_X(t,u),u\rangle \leq - K\norm{u}^{r+1}_{L^{r+1}} + C \norm{u}^2_{L^2} + C \norm{X(t)}^2_{L^2} + C \norm{X(t)}^{r+1}_{L^{r+1}},
    $
\end{enumerate}
where $\text{Lip}_f$ denotes the one-sided Lipschitz constant of $f$.
\end{lemma}
\begin{lemma}\label{exvarnag}
Suppose that $X \in L^2([0,T];L^2) \cap L^{r+1}([0,T];L^{r+1})$. Then there exists a unique variational solution $$w \in C([0,T];L^2) \cap L^{r+1}([0,T];L^{r+1})\cap L^2([0,T];\mathcal B^2_{1/2})$$ of the equation $$\partial_t w(t) = A w(t) + F_X(t,w(t))$$ with respect to the Gelfand triple $\mathcal B^2_{1/2} \hookrightarrow L^2 \hookrightarrow (\mathcal B^2_{1/2})^\ast$. In particular, there exists a unique variational solution of equation \eqref{varNag}.
\end{lemma}
\begin{proof}
Even though identities (1) and (2) from Lemma \ref{moncoer} entail the necessary hemicontinuity, monotonicity and coercivity conditions (H2) and (H3) of e.g. Theorem 5.1.3 in \cite{LiuRoeckner2015}, we cannot apply that theorem directly, since for $f(x) \asymp -x^r$ with $r > 3$, the growth condition (H4') is not satisfied.

However, by Thm. 3.1.1. in \cite{LiuRoeckner2015}, these properties suffice to obtain the existence of finite dimensional Galerkin approximations $(w_n)_{n \in \mathbb N}$ with respect to some orthonormal basis of $L^2$. Further, by inequality \ref{moncoer}.(2), there exists a constant $C$ only dependent on $$\int_0^T  \norm{X(t)}^2_{L^2} + \norm{X(t)}^{r+1}_{L^{r+1}} \d t$$ such that $$\sup_{t \in [0,T]} \norm{w_n(t)}^2_{L^2} + \int_0^T \norm{w_n(t)}^{r+1}_{L^{r+1}} \d t + \int_0^T \norm{(-A)^{1/2}w_n(t)}^2_{L^2} \d t \leq C.$$ With this stronger a priori inequality, we can compensate for the lack of direct bounds on the nonlinearity $F_X$. We  observe that the bound $$\norm{F_X(-,w_n(-))}_{(L^{r+1}([0,T];L^{r+1}))^\ast} = \norm{F_X(-,w_n(-)}_{L^\frac{r+1}{r}([0,T];L^\frac{r+1}{r})} \leq K$$ follows for some $K$ independent of $n$. Hence we can extract weakly convergent subsequences and, after minor modifications of the Lions-Magenes lemma (Lemmas 4.2.5. and 4.2.6 in \cite{LiuRoeckner2015}), we can imitate the proof of Thm. 4.2.4 in \cite{LiuRoeckner2015} to show that the weak limit is in fact a variational solution of equation \eqref{varNag}.

Uniqueness follows by a Grönwall argument after utilising the monotonicity property of the nonlinearity $F_X$.
\end{proof}
\begin{proof}[Existence of solutions of equation \eqref{Nag}]
The obtained variational solutions to equation \eqref{varNag} are in fact mild solutions; the verification is standard (cf. \cite{hairer2023introductionstochasticpdes}, Ch. 5) and we omit it for the sake of brevity. Therefore we find that almost surely, $$w(t) = S(t)U_0 + \int_0^t S(t-s)(F_{N_A}(s,w(s))) \d s.$$ 
Let $$V_0 \in v^{TW}_0 + L^2$$ be given and $w(t)$ be the mild solution of equation \eqref{varNag}. Summing up, we see that $V(t) \coloneqq w(t) + v^{TW}(t) + \varepsilon N_A(t)$ satisfies the equation $$V(t) - v^{TW}(t) = S(t)V_0 + \int_0^t S(t-s)\left(f(V(s))-f(v^{TW}(s)\right)\d s +  \varepsilon \int_0^t S(t-s)\d N(s).$$ Therefore, the process $V(t)$ is a mild solution (cf. Definition \ref{SolDef}) of equation \eqref{Nag}. As $V - v^{TW} = w + \varepsilon N_A$, the regularities of $w$ and $N_A$ imply that $$V-v^{TW} \in C([0,T];L^2) \cap L^{r+1}([0,T];L^{r+1}) \cap L^2([0,T];\mathcal B^2_\delta)$$ for any $0 < \delta < \eta - \gamma$. Uniqueness of the solution which satisfies $$V-v^{TW} \in C([0,T];L^2) \cap L^{r+1}([0,T];L^{r+1})$$ follows by Lemma \ref{exvarnag}, as $$w \coloneqq V - v^{TW} - \varepsilon N_A$$ can then be shown to be a variational solution of equation \eqref{varNag} for $X = \varepsilon N_A$.
\end{proof} 
\section{Pathwise stability for small perturbations} \label{pathstab}
In this section, we aim to prove pathwise stability results for travelling waves perturbed by Hölder continuous noise. In this context, stability is measured by distance to some spatial translate of the travelling wave shape $v_0^{TW}$. 
Let $\Gamma \coloneqq \Set{v^{TW}_0(\cdot + \phi\nu);\phi \in \mathbb R}$ denote the set of travelling wave fronts and $V$ be a mild solution of equation \eqref{Nag} with $V(0) = v^{TW}_0$ (understood in the sense of Definition \ref{SolDef}). For $u \in v^{TW}_0 + L^2$, define $$d(u, \Gamma) = \inf_{\phi \in \mathbb R} \norm{u - v^{TW}_\phi}_{L^2(\mathcal O)} < \infty.$$ 
\begin{definition}
We say that the travelling wave solution $v^{TW}$ is \textit{stable} under the influence of small noise amplitudes if, given any $\delta > 0$, the solution $V$ satisfies $$\sup_{0 \leq t\leq T}d(V(t),\Gamma) < \delta$$ for any small enough noise amplitude $\varepsilon > 0$.
\end{definition}

\subsection{Deterministic phase adaptation}
In this subsection, we derive existence of a real-valued stochastic process $C(t)$ which approximates a phase $\phi$ for which $\norm{u - v^{TW}_\phi}_{L^2(\mathcal O)}$ is minimal.
We follow \cite{StannatTW} and \cite{KruegerStannat} and introduce a gradient-descent type ODE into the direction of local minima of $$C \mapsto \norm{V(t) - v^{TW}_C}^2_{L^2},$$  where $\nu$ denotes the direction of wave propagation. This means that we update our estimates $C_n$ via $$C_{n+1} - C_n =  m(t_{i+1}-t_i) \left \langle V(t_i) - v^{TW}_{C_n}, \nu \cdot \nabla v^{TW}_{C_n} \right \rangle.$$ If we take the limit $\Delta t \to 0$ and account for the wave speed $c$, we arrive at the ordinary differential equation \begin{equation}\label{GradDesc}
C'(t) = c + m\langle V(t) - v^{TW}_{C(t)}, \nu \cdot \nabla v^{TW}_{C(t)} \rangle
\end{equation} with $C(0) = 0$ and $m > 0$. 

\begin{proposition}
For every solution $V(t)$ of equation \eqref{Nag} that satisfies $$V-v^{TW} \in C([0,T];L^2) \cap L^{r+1}([0,T];L^{r+1})\cap L^2([0,T];\mathcal B^2_\delta) $$ for some $\delta > 0$, there exists a unique solution $C \colon [0,T] \rightarrow \mathbb R$ of equation \eqref{GradDesc}.
\end{proposition}

\begin{proof}
The proof proceeds by means of a Picard iteration for successive short enough subintervals of $[0,T]$, i.e. one shows that the map $F \colon C([t_1,t_2]) \rightarrow C([t_1,t_2])$ with $$F(h) =  C_0 + \int_{t_1}^\cdot c + m\left \langle V(s) - v^{TW}_{h(s)}, \nu \cdot \nabla v^{TW}_{h(s)} \right \rangle \d s$$ is a strict contraction on $C([t_1,t_2])$ for arbitrary $C_0 \in \mathbb R$ and $|t_2-t_1|$ small enough. 
Observe that by translation invariance, $$\left\langle v^{TW}_{cs} - v^{TW}_{h(s)}, \nu \cdot \nabla v^{TW}_{h(s)} \right \rangle = \left \langle v^{TW}_{cs +(cs - h(s))} - v^{TW}_{cs}, \nu \cdot \nabla v^{TW}_{cs} \right \rangle,$$
and hence
$$
\begin{aligned}
F(h)(t)-F(g)(t) &= m\int_{t_1}^{t} \left\langle V(s) - v^{TW}_{cs}, \nu \cdot \nabla (v^{TW}_{h(s)} - v^{TW}_{ g(s)})\right\rangle \d s \\ &\quad -m\int_{t_1}^{t} \left\langle v^{TW}_{cs+(cs-h(s))} - v^{TW}_{cs +(cs-g(s))}, \nu \cdot \nabla v^{TW}_{cs} \right\rangle \d s.
\end{aligned}
$$
Using 
$$
\norm{v^{TW}_{2cs - h(s)} - v^{TW}_{2cs -g(s)}}_{L^2} \leq \norm{\nu \cdot \nabla v^{TW}_0}_{L^2} 
|h(s) - g(s)| 
$$
and Assumption \ref{Intgrbl} on $\nu\cdot\nabla v^{TW}_0$, we obtain global Lipschitz continuity of 
the second term. 

To obtain the Lipschitz continuity of the first term we can similarly estimate 
$$
\norm{\nu\cdot\nabla (v^{TW}_{2cs - g(s)} - v^{TW}_{2cs -h(s)})}_{L^2}  
\leq \norm{\nu \cdot H(v^{TW}_0) \cdot \nu}_{L^2} |g(s) - h(s)| 
$$
using the Hessian of $v^{TW}_0$ in the case $c=0$. 

In the case where $c \neq 0$ we can drop the assumption on the Hessian, using that
$$ 
\begin{aligned}
\left\langle V(s) - v^{TW}_{cs}, \nu \cdot \nabla (v^{TW}_{h(s)} - v^{TW}_{ g(s)})\right \rangle  
&= -\frac 1c \left\langle V(s) - v^{TW}_{cs}, A (v^{TW}_{h(s)} - v^{TW}_{ g(s)}) + 
\left(f(v^{TW}_{h(s)})-f(v^{TW}_{g(s)})\right)\right\rangle\\
&= \frac 1c \left\langle (-A)^\delta (V(s) - v^{TW}_{cs}), (-A)^{1-\delta} (v^{TW}_{h(s)} - v^{TW}_{ g(s)})\right\rangle \\ &\quad -\frac 1c \left\langle V(s) - v^{TW}_{cs},
f(v^{TW}_{h(s)})-f(v^{TW}_{g(s)})\right\rangle
\end{aligned}  
$$
for $0 < \delta < \eta - \gamma$. Now, again by Assumption \ref{Intgrbl}, we get the Lipschitz estimate$$\norm{f(v^{TW}_{h(s)})-f(v^{TW}_{g(s)})} \leq \norm{f'(v^{TW}_0)}_{L^\infty(\mathcal O)}\norm{\nu \cdot \nabla v^{TW}_0}_{L^2(\mathcal O)}|h(s)-g(s)|.$$ It remains to find a Lipschitz estimate on $$\norm{(-A)^{1-\delta} (v^{TW}_{h(s)} - v^{TW}_{g(s)})}_{L^2}.$$ Then, since $V - v^{TW} \in L^2([0,T];\mathcal B^2_\delta)$ for any $0 < \delta < \eta - \gamma$, we can prove that $F$ in fact defines a strict contraction on $C([0,T])$.
By definition of a travelling wave solution, $$\partial_t v^{TW}_{x + ct} = Av^{TW}_{x + ct} + f(v^{TW}_{x + ct})$$ for any $x \in \mathbb R$. This in turn implies that for arbitrary $x_1$ and $x_2$, the difference $v^{TW}_{x_1+ct} - v^{TW}_{x_2+ct}$ satisfies the partial differential equation $$\partial_t(v^{TW}_{x_1+ct} - v^{TW}_{x_2+ct}) = A(v^{TW}_{x_1+ct} - v^{TW}_{x_2+ct}) + f(v^{TW}_{x_1+ct}) - f(v^{TW}_{x_2+ct}) .$$ Using Assumption \ref{AnaSem}, we can therefore find the mild representation $$v^{TW}_{x_1+ct} - v^{TW}_{x_2+ct} = S(t)(v^{TW}_{x_1} - v^{TW}_{x_2})+\int_0^t S(t-s)\left(f(v^{TW}_{x_1+cs}) - f(v^{TW}_{x_2+cs})\right)\d s$$ and in particular,$$v^{TW}_{x_1} - v^{TW}_{x_2} = S(1)(v^{TW}_{x_1-c} - v^{TW}_{x_2-c}) + \int_0^1 S(1-s)\left(f(v^{TW}_{x_1-c(1-s)}) - f(v^{TW}_{x_2-c(1-s)})\right) \d s.$$ Now Bochner's inequality and the generic semigroup estimate $$\norm{(-A)^s S(t) u}_{L^2} \leq M \frac{\norm{u}_{L^2}}{t^s},~ u \in L^2$$ imply that $$\norm{(-A)^{1-\delta}(v^{TW}_{x_1} - v^{TW}_{x_2})}_{L^2} \leq M \norm{\nu \cdot \nabla v^{TW}_0}_{L^2(\mathcal O)} \left(1+\frac{\norm{f'(v^{TW}_0)}_{L^\infty(\mathcal O)}}{\delta}\right)|x_1-x_2|.$$ Inserting $x_1 = h(s)$ and $x_2 = g(s)$ and choosing $|t_1-t_2|$ small enough now finishes the proof.
\end{proof}

\subsection{First order approximation of residual}
In the following, let $C(t)$ denote the unique solution of the phase-adaptation ODE \eqref{GradDesc} and set $$\tilde v^{TW}(t) = v^{TW}_{C(t)}.$$ Our aim now is to analyse the residual process $\tilde U(t) \coloneqq V(t) - \tilde v^{TW}(t).$ To demonstrate that the fluctuations of the paths of the driver $N$ dominate the dynamics of the error term $\tilde U(t)$ for small noise amplitudes, we decompose the error into two terms $$\tilde U(t) = \varepsilon Z_\varepsilon(t) + y_\varepsilon(t).$$ Here, heuristically, $Z_\varepsilon$ denotes an Ornstein-Uhlenbeck process which approximates $\tilde U$ to the first order, and $y_\varepsilon$ denotes the nonlinear residual.

To state the main result of this subsection, we need some preliminary definitions. We introduce the family of linear operators $$A(t) \colon D(A) \subset L^2(\mathcal O) \rightarrow L^2(\mathcal O)$$ 
defined by $$A(t)u \coloneqq Au + f'\left(\tilde v^{TW}(t) \right)u -\tilde P(t)u,$$ where $$\tilde P(t) u \coloneqq m\langle u, \nu \cdot \nabla \tilde v^{TW}(t)\rangle \nu \cdot \nabla \tilde v^{TW}(t)$$ denotes the scaled projection onto $\nu \cdot \nabla \tilde v^{TW}(t)$ for some arbitrary $m \geq C_\ast$ (cf. Assumption \ref{SpcGp}). Further, let $E(t,s)$ denote the evolution system generated by the family $(A(t))_{t \in [0,T]}$ (cf. Proposition \ref{EvSysGen} and the preceding definition). We also introduce the nonlinear residual term $$
\tilde R(t, u) \coloneqq f\left(u + \tilde v^{TW}(t)\right) - f\left(\tilde v^{TW}(t)\right) - f'\left(\tilde v^{TW}(t)\right)u.$$ Finally, let the auxilliary process $N_{A-\lambda}(t) \coloneqq \int_0^t e^{-\lambda(t-s)} S(t-s) \d N(s)$ denote a mild solution of a damped heat equation with damping parameter $\lambda \geq 0$, perturbed by the path $N$.

\begin{proposition} \label{DecompositionMain}
The residual $\tilde U$ decomposes as $\tilde U = \varepsilon Z_\varepsilon + y_\varepsilon$, where
$$Z_\varepsilon(t) \coloneqq \int_0^t E(t,s)\left(\tilde P(s) + f'(v^{TW}(s)) + \lambda \right) N_{A-\lambda}(s) \d s + N_{A-\lambda}(t)$$ and $y_\varepsilon \coloneqq \tilde U - \varepsilon Z_\varepsilon$ solves the variational equation $$\partial_t y_\varepsilon(t) = A(t) y_\varepsilon(t) + R(t,y_\varepsilon(t) +\varepsilon Z_\varepsilon(t))$$ with respect to the Gelfand triple $\mathcal B^2_{1/2} \hookrightarrow L^2 \hookrightarrow (\mathcal B^2_{1/2})^\ast$.
\end{proposition}

\begin{remark}
We note that the choice of $\lambda$ does not alter the resulting decomposition: Since $y_\varepsilon + \varepsilon Z_\varepsilon(t) = \tilde U(t)$, $y_\varepsilon$ can be seen to be independent of $\lambda$ and thus $Z_\varepsilon$ too. However, different choices of $\lambda$ enable different kinds of estimates in subsequent sections.
\end{remark}

As an intermediate step to reach this result, we decompose $\tilde U$ as $\tilde U = \tilde w_\lambda + \varepsilon N_{A-\lambda}$.  Subsequently, we define  $\varepsilon Z_\varepsilon \coloneq \tilde v_\lambda + \varepsilon N_{A-\lambda}$, where $\tilde v_\lambda$ denotes a first order approximation of $\tilde w_\lambda$. Then $y_\varepsilon \coloneqq \tilde w_\lambda - \tilde v_\lambda = \tilde U - \varepsilon Z_\varepsilon$ defines the higher order residual. 

It is therefore central to identify the first-order approximation $\tilde v_\lambda$ of $\tilde w_\lambda = \tilde U - \varepsilon N_{A-\lambda}.$ The crucial insight in this regard is that $\tilde w_\lambda$ solves a semilinear variational equation whose linear part $(A(t))_{t \in [0,T]}$ generates an exponentially decaying evolution system. The solution $\tilde v_\lambda$ of the corresponding inhomogeneous linear equation therefore provides a faithful approximation of $\tilde w_\lambda$ for small initial data. 

\begin{remark}
Note that the above expression for $Z_\varepsilon(t)$ contains a pathwise notion of convolution of evolution systems generated by bounded perturbations of sectorial operators against Hölder continuous paths. For the specific class of evolution systems that we consider, this can be considered as a simple extension of both the framework of \citet{PronkVeraar}, where Wiener noise is considered, and \citet{Gubinelli2004YoungIA}, where the generator $A$ is assumed to be constant. 
\end{remark}
\begin{definition}
Let $L$ be an injective sectorial operator on some Banach space $\mathcal B$ and $$X \in C^\eta([0,T];\mathcal B_{-\gamma})$$ be some Hölder continuous path with $\eta > \gamma$. Given $R \in L^\infty([0,T];L(\mathcal B_\delta,\mathcal B))$ for some $\delta \in [0,\eta - \gamma)$, let $P(t,s)$ denote the evolution system generated by $L(t) \coloneqq L + R(t)$. We then define $$\int_0^t P(t,s) \d X(s) \coloneqq \int_0^t P(t,s) R(s) X_L(s) \d s + X_L(t),$$ where $X_L(t)$ denotes the convolution of $X$ against the semigroup generated by $L$.
\end{definition}
In our specific case, $A(t)$ is the propagating family of operators and $A-\lambda$ is a sectorial operator such that $B_\lambda(t) \coloneqq A(t) - (A-\lambda)$ is uniformly bounded in $t$ as an operator on $L^2(\mathcal O)$. Then we can define $$\int_0^t E(t,s) \d N(s) \coloneqq \int_0^t E(t,s)B_\lambda(s) N_{A-\lambda}(s) \d s + N_{A-\lambda}(t).$$ Note that this definition is independent of $\lambda \geq 0$.

We now collect some results on the auxilliary processes $N_{A-\lambda}$. In the subsequent section, we will study how the error term $\tilde U(t)$ grows in relation to the fluctuations of $N_{A-\lambda}$. To achieve this, we need the following integration by parts formula, first shown to hold almost surely for all $t \in [0,T]$ in the case of $Q$-fractional Brownian motion with Hurst parameter $H > \frac12$ in \cite{SchmalfussMaslowski}, and later extended by \citet{MaslowskiPospisilErgodicity} to all $H \in (0,1)$. We defer the proof to the appendix.

\begin{proposition} \label{ibp}
Let $S(t)$ be an analytic semigroup on a Banach space $\mathcal B$ with injective generator $A$, and $N \in C^\eta_0([0,T];\mathcal B_{-\gamma})$ for $\gamma < \eta$, where $\mathcal B_\rho = D((- A)^\rho)$. Then the $C^\kappa_0([0,T];B_\delta)$-valued convolution $N_A$ satisfies the identity \begin{equation} \label{ibpformula}
\int_0^t S(t-s) \d N(s) = \int_0^t A S(t-s)(N(s)-N(t)) \d s +S(t)N(t),
\end{equation}where $0 < \delta < \eta - \gamma$ and $0 < \kappa < \min{(\gamma - \eta - \delta, 1)}$.
\end{proposition}
By application of this integration by parts formula to the the convolution $N_{A-\lambda}$ together with Fubini's theorem and standard semigroup manipulations, we can derive the following identity.
\begin{proposition}\label{DampDuham} Let $N$, $A$ be chosen as in Proposition \ref{ibp} and and $\lambda \geq 0$ be arbitrary. For any $t \in [0,T]$,
$$N_{A-\lambda}(t) = N_A(t) - \lambda \int_0^t S(t-s) N_{A-\lambda}(s) \d s.$$
\end{proposition}
Using Proposition \ref{DampDuham}, we can subtract the terms with unbounded variation from $\tilde U(t)$ by subtracting the process $\varepsilon N_{A-\lambda}$ from $\tilde U(t)$. 
We now identify the semilinear equation satisfied by the resulting term, $\tilde w_\lambda = \tilde U - \varepsilon N_{A-\lambda}$.

\begin{proposition}
The process $$\tilde w_\lambda(t) \coloneqq V(t) - \tilde v^{TW}(t) - \varepsilon N_{A-\lambda}(t) = \tilde U(t) - \varepsilon N_{A-\lambda}(t)$$ satisfies the equation $$\begin{aligned}
\partial_t \tilde w_\lambda(t) &= A(t) \tilde w_\lambda(t)  + \varepsilon \left( \tilde P(t) + f'\left(\tilde v^{TW}(t)\right) + \lambda\right) N_{A-\lambda}(t)  +\tilde R(t,\tilde w_\lambda(t) + \varepsilon N_{A-\lambda}(t)) 
\end{aligned}$$ with respect to the Gelfand triple $\mathcal B^2_{1/2} \hookrightarrow L^2 \hookrightarrow (B^2_{1/2})^\ast$, where $w_\lambda(0) = u_0.$
\end{proposition}
\begin{proof} Using the decomposition
$$
\begin{aligned}
\tilde w_\lambda(t) = \underbrace{V(t) - v^{TW}(t) - \varepsilon N_A(t)}_{=w(t)} + v^{TW}(t) - \tilde v^{TW}(t)  + \varepsilon \lambda \int_0^t S(t-s)N_{A-\lambda}(s) \d s
\end{aligned}
$$the claim can be verified by direct calculations by means of the partial differential equations satisfied by each term.
\end{proof}

Finally, we show that \( (A(t))_{t \in [0,T]} \) generates an exponentially decaying evolution system \( E \) on \( L^2 \) and we establish a decomposition of the process \( \tilde{w}_\lambda(t) \) into a sum $\tilde w_\lambda = \tilde v_\lambda + y_\varepsilon$ of a convolution $\tilde v_\lambda$ against an $E$, with a nonlinear residual $y_\varepsilon$. With this representation, we can leverage the exponential decay properties of $E$, enabling us to derive upper bounds on the norm of $\tilde{U}(t) = \tilde{w}_\lambda(t) + \varepsilon N_{A-\lambda}(t).$

\begin{definition}[Evolution system] 
Let $\mathcal B$ be a separable Banach space and $\Delta^2 \coloneqq \{(t,s) \in \mathbb [0,T]^2: s<t \}$. An evolution system on $[0,T]$ is defined as a map $P \colon \Delta^2 \rightarrow L(\mathcal B)$ such that
\begin{itemize}
    \item $P(t,t) = I$.
    \item $P(t,s)P(s,r) = P(t,r)$.
    \item $(t,s) \mapsto P(t,s)$ is strongly continuous.
\end{itemize}
\end{definition}
\begin{proposition}[Existence in linear perturbation case \cite{henry81:GTS}] \label{EvSys}
Let $L$ be a sectorial operator on $\mathcal B$ and $R \in L^\infty([0,T];L(\mathcal B_\alpha, \mathcal B))$ with $0 \leq \alpha < 1$. Then the operators $L(t) \coloneqq L + R(t)$ generate an evolution system $P$ such that $x(t) \coloneqq P(t,\tau)x$ solves the nonautonomous evolution equation $$\partial_t x(t) = (L + R(t))x(t)$$ with $x(\tau) = x$.
\end{proposition}

\begin{proposition} \label{EvSysGen}
The family of operators $(A(t))_{t\in[0,T]}$ generates an exponentially decaying evolution system $E$.
\end{proposition}
\begin{proof}
Since $A$ is sectorial and \( f'(\tilde{v}^{TW}(t)) \) and \( \tilde{P}(t) \) are bounded as operators mapping $L^2$ into itself, we see that the family \( (A(t))_{t \in [0,T]} \) satisfies the conditions of Proposition \ref{EvSys}.

Importantly, the operators $A(t)$ are dissipative. This can be seen through the Poincaré-type inequality 
$$\langle Au + f'(\tilde v^{TW})u, u \rangle \leq -\kappa_\ast \norm{u}^2_{\mathcal B^2_{1/2}} + C_\ast \langle u, \nu \cdot \nabla \tilde v^{TW}(t) \rangle ^2,$$ which holds by Assumption \ref{SpcGp}. Rearranging this inequality yields that $$\langle Au + f'(\tilde v^{TW})u - C_\ast \langle u, \nu \cdot \nabla  \tilde v^{TW}\rangle \nu \cdot \nabla \tilde v^{TW}, u \rangle \leq -\kappa_\ast \norm{u}^2_{1/2},$$ implying in particular that $A(t)$ is dissipative for $m \geq C_\ast$. It follows that $E$ is exponentially decaying.
\end{proof}

\begin{proof}[Proof of Proposition \ref{DecompositionMain}]
Finally, we can define $$\tilde v_\lambda (t) \coloneqq \varepsilon \int_0^t E(t,s) \left(\tilde P(s) + f'(v^{TW}(s)) + \lambda \right) N_{A-\lambda}(s)  \d s.$$
We remark that by the integrability properties of $N_{A-\lambda}$, $\tilde v_t$ solves the variational equation $$\partial_t \tilde v_\lambda(t) = A(t)\tilde v_\lambda(t) + \left( \tilde P(t) + f'\left(\tilde v^{TW}(t)\right) + \lambda\right) N_{A-\lambda}(t)$$ with respect to the Gelfand triple $\mathcal B^2_{1/2} \hookrightarrow L^2 \hookrightarrow (B^2_{1/2})^\ast$. As a consequence, the process $y_\varepsilon \coloneqq \tilde w_\lambda - \tilde v_\lambda$ solves the variational equation $$\partial_t y_\varepsilon(t) = A(t) y_\varepsilon(t) + R(t,y_\varepsilon(t) + \tilde v_\lambda(t) + \varepsilon N_{A-\lambda}(t)).$$
By definition of these processes,  the process $\tilde w_\lambda$ decomposes as $$\begin{aligned}
\tilde w_\lambda(t) &= y_\varepsilon(t) +\varepsilon \int_0^t E(t,s) \left(\tilde P(s) + f'(v^{TW}(s)) + \lambda \right) N_{A-\lambda}(s)  \d s.
\end{aligned}$$
Therefore, we also obtain a decomposition of $\tilde U(t) = V(t) - \tilde v^{TW}(t)$, which yields
$$
\tilde U(t) = \tilde w_\lambda(t) + \varepsilon N_{A-\lambda}(t) = y_\varepsilon(t) +\varepsilon \int_0^t E(t,s) \left(\tilde P(s) + f'(v^{TW}(s)) + \lambda \right) N_{A-\lambda}(s) \d s + \varepsilon N_{A-\lambda}(t). $$
\end{proof}

\subsection{Stability for small noise amplitudes}
Recall that we defined $\mathcal U = L^2\cap L^{r+1}$. 
In the remainder of Section \ref{pathstab}, we assume that $V(0) = v_0^{TW}$, or equivalently, $u_0 = 0$. The main result of this subsection is an upper bound of the form $$\sup_{0 \leq t \leq T} \norm{y_\varepsilon(t)}_{L^2} \in o\left(\varepsilon \sup_{0 \leq t \leq T}\norm{Z_\varepsilon(t)}_{\mathcal U}\right)$$ on the nonlinear residual $y_\varepsilon(t)$ for sufficiently small noise amplitudes $\varepsilon>0$. 

\begin{proposition} \label{Ybound}
Let the constant $C_y > 0$ be as in Lemma \ref{diffineq} and $\kappa_\ast > 0$ be the dissipativity constant of $A(t)$ given by \eqref{SpcGpIneq}. Then there exists a constant $z_\ast > 0$ independent of $\varepsilon > 0$ and $T>0$ such that whenever $$\sup_{0 \leq t \leq T} \norm{Z_\varepsilon(t)}_{\mathcal U} \leq \frac{z_\ast}{\varepsilon},$$ then $$ \sup_{t \in [0,T]}\norm{y_\varepsilon(t)}^2_{L^2}\leq \frac{2C_y}{\kappa_\ast}\sum_{k=3}^{r+1} \varepsilon^k \sup_{0 \leq t \leq T} \norm{Z_\varepsilon(t)}_{\mathcal U}^{k}.$$ \end{proposition}

The proof of Proposition \ref{Ybound} hinges on the following differential inequality. 

\begin{lemma} 
\label{diffineq}
There are constants $K_y, C_y >0$ independent of $\varepsilon > 0$ and exponents $2 < p_3 < \dots < p_r$ such that the $L^2$-norm of the remainder $y_\varepsilon(t)$ satisfies the differential inequality $$\begin{aligned}
\partial_t \norm{y_\varepsilon(t)}^2_{L^2} &\leq - \kappa_\ast \norm{y_\varepsilon(t)}^2_{L^2} + K_y \sum_{k = 3}^{r} \norm{y_\varepsilon(t)}^{p_k}_{L^2} + C_y\sum_{k=3}^{r+1} \varepsilon^k \sup_{0 \leq t \leq T}\norm{Z_\varepsilon(t)}_{\mathcal U}^{k}
\end{aligned}$$ for all $t \in [0,T].$
\end{lemma}
\begin{remark} \label{f0equal0}
If $f_0 \equiv 0$, then the results of this section still hold under mild modifications of assumptions and statements. Namely, the differential inequality in Lemma \ref{diffineq} changes to $$\partial_t \norm{y_\varepsilon(t)}^2_{L^2} \leq - \kappa_\ast \norm{y_\varepsilon(t)}^2_{L^2} + K_y \norm{y_\varepsilon(t)}^{p_1}_{L^2} + C_y \varepsilon^3 \sup_{0 \leq t \leq T}\norm{Z_\varepsilon(t)}_{L^3}^3,$$ under the assumption that $N \in C^\eta_0([0,T];\mathcal B^2_{-\gamma} \cap \mathcal B^3_{-\gamma})$. Consequently, we obtain the upper bound $$\sup_{t \in [0,T]}\norm{y_\varepsilon(t)}^2_{L^2}\leq \frac{2C_y}{\kappa_\ast}\varepsilon^3 \sup_{0 \leq t \leq T} \norm{Z_\varepsilon(t)}_{L^3}^{3}.$$ Subsequent bounds, i.e. Theorems \ref{ShortLongBounds}, \ref{mainthm}, \ref{LongboundSelfSim} and Corollaries \ref{obound}, \ref{AppliedSelfSim}, and hold with $r+1 $ replaced by $3$ as upper summation index, $\mathcal U = L^2 \cap L^3$ and $\mathcal U_{-\gamma} = \mathcal B^2_{-\gamma} \cap \mathcal B^3_{-\gamma}$.
\end{remark}
\begin{remark}
If $u_0 \neq 0$, so that the initial condition is a perturbed travelling wave, it is clear that $\sup_{0 \leq t \leq T} \norm{y_\varepsilon(t)}_{L^2}$ cannot approach $0$. In this case, as $\varepsilon \to 0$, we would obtain the fixed upper bound $\norm{V_0-v^{TW}_0}_{\bm L^2}$ on $\norm{y_\varepsilon}_{L^2}$ and hence $\sup_{0 \leq t \leq T} \mathrm{d}(V(t),\Gamma) \leq  \varepsilon \sup_{0 \leq t \leq T} \norm{Z_\varepsilon(t)}_{L^{2}} + \norm{V_0-v^{TW}_0}_{\bm L^2}$. Although these bounds can be improved on by the above lemma, we generally lose information about higher-order errors. 
\end{remark}
\begin{proof}
We rely on the fact that the process $y_\varepsilon$ is a variational solution of the partial differential equation $$\partial_t y_\varepsilon(t) = A(t)y_\varepsilon(t) + \tilde R(t, y_\varepsilon(t) + \varepsilon Z_\varepsilon(t)).$$ Here, $$\begin{aligned}
\tilde R(t, u) &= \sum_{k=2}^r f_0^{(k)}(\tilde v^{TW}(t))u^k + f_1(u + \tilde v^{TW}(t)) - f_1(\tilde v^{TW}(t)) - f'(\tilde v^{TW}(t))u,
\end{aligned}$$
where $f_0$ is an odd-order polynomial with negative leading order coefficient and $f_1$ is twice differentiable with bounded first and second derivative.
Note that by Taylor's theorem there exists a measurable function $ \xi \colon [0,T] \times \mathbb R \rightarrow \mathbb R$ with $$\begin{aligned}
&f_1(y_\varepsilon(t) + \varepsilon Z_\varepsilon(t) + \tilde v^{TW}(t)) - f_1(\tilde v^{TW}(t)) - f_1'(\tilde v^{TW}(t))(y_\varepsilon(t) + \varepsilon Z_\varepsilon(t)) \\&= f_1''(\xi(t))(y_\varepsilon(t) + \varepsilon Z_\varepsilon(t))^2.
\end{aligned}$$
As $f_0$ has a negative leading coefficient, $f_0^{(r)} \equiv -a$ for some $a > 0$ and hence $$\begin{aligned}
\tilde R(t,y_\varepsilon(t) + \varepsilon Z_\varepsilon(t)) = &-a(y_\varepsilon(t) + \varepsilon Z_\varepsilon(t))^r + (f_1''(\xi(t)))(y_\varepsilon(t) + \varepsilon Z_\varepsilon(t))^2 \\ &+ \sum_{k=2}^{r-1} f_0^{(k)}(\tilde v^{TW}(t))(y_\varepsilon(t) + \varepsilon Z_\varepsilon(t))^k
\end{aligned}$$
The integrability properties $y_\varepsilon$ inherits from the solution $V$ of the equation \eqref{Nag} ensure that the Lions-Magenes lemma is applicable to this process. It follows that 
$$\begin{aligned}
\partial_t \norm{y_\varepsilon(t)}^2_{L^2} &= -\kappa_\ast \norm{y_\varepsilon(t)}^2_{\mathcal B^2_{1/2}} + \langle \tilde R(t, y_\varepsilon(t) + \varepsilon Z_\varepsilon(t)),y_\varepsilon(t) \rangle \\ &= - \kappa_\ast \left( \norm{(-A)^{1/2}y_\varepsilon(t)}^2_{L^2} + \norm{y_\varepsilon(t)}^2_{L^2}\right) + \langle \tilde R(t, y_\varepsilon(t) + \varepsilon Z_\varepsilon(t)),y_\varepsilon(t) \rangle
\end{aligned}$$
To arrive at the desired bound, we expand the polynomial terms and apply Young's inequality for products and the interpolation equality \eqref{IntPlIneq} to the residual $$\begin{aligned}
\left|\langle \tilde R\left(t,y_\varepsilon(t) + \varepsilon Z_\varepsilon(t)\right),y_\varepsilon(t) \rangle \right| &\leq  - a \int (y_\varepsilon(t) + \varepsilon Z_\varepsilon(t))^r y_\varepsilon(t) \d x \\ & \quad +\int |f_1''(\xi(t))|(|y_\varepsilon(t)|+\varepsilon |Z_\varepsilon(t)|)^2 |y_\varepsilon(t)|\d x
\\ & \quad + \sum_{k=2}^r \int |f^{(k)}_0(\tilde v^{TW}(t))| (|y_\varepsilon(t)|+\varepsilon |Z_\varepsilon(t)|)^k |y_\varepsilon(t)| \d x
\\ &= I + II + III, \text{ say.}
\end{aligned}$$ By boundedness of $f_1''$ and $f^{(k)}(\tilde v^{TW}(t))$ combined with convexity of $x^p$, $p \geq 1$, we find that 
$$\begin{aligned}
II + III &\leq C \sum_{k=2}^{r-1} \norm{y_\varepsilon(t)}^{k+1}_{L^{k+1}} + \langle \varepsilon^k |Z_\varepsilon(t)|^k, |y_\varepsilon(t)| \rangle\\
& \leq C \sum_{k=2}^{r-1} \norm{y_\varepsilon(t)}_{L^{k+1}}^{k+1}  + \varepsilon^{k+1}\norm{Z_\varepsilon(t)}_{L^{k+1}}^{k+1} \\
&\leq C \sum_{k=3}^{r} \norm{(-A)^{1/2}y_\varepsilon(t)}^{k \theta_k}_{L^2} \norm{y_\varepsilon}^{k(1-\theta_k)}_{L^2}  + C\sum_{k=3}^{r} \varepsilon^k\norm{Z_\varepsilon(t)}_{L^{k}}^{k} .
\end{aligned}$$
Here, $C$ denotes a constant that changes from line to line. Now, by Assumption \ref{IntPl}, $k \theta_k < 2$ for $3 \leq k \leq r $. Therefore, we can again apply Young's inequality with $p = \frac{2}{k \theta_k}$ and $q = \frac{2}{2-k \theta_k}$ and see that for suitably chosen $0 < \kappa < \kappa_\ast$ and $K_y > 0$ dependent on $\kappa$ and $r$,
$$
\begin{aligned}
II + III &\leq \kappa \norm{(-A)^{1/2}y_\varepsilon(t)}^2_{L^2} + K_y \sum_{k=3}^{r} \norm{y_\varepsilon(t)}^{k(1-\theta_k)\frac{2}{2-k \theta_k}}_{L^2} \\
&\quad + \underbrace{C}_{\eqqcolon C_{1,y}}\sum_{k=3}^{r} \varepsilon^k\norm{Z_\varepsilon(t)}_{L^{k}}^{k}.
\end{aligned}$$
Note that if $k \theta_k < 2$, then $k(1-\theta_k)\frac{2}{2-k \theta_k} > 2$ if and only if $k > 2$, so that we can be assured that all exponents of $\norm{y_\varepsilon(t)}_{L^2}$ are larger than $2$. 
It is left to finish the estimate 
$$\begin{aligned}
&I = - a \int (y_\varepsilon(t) + \varepsilon Z_\varepsilon(t))^r y_\varepsilon(t) \d x \\& \leq -a \norm{y_\varepsilon(t)}^{r+1}_{L^{r+1}}+ C a\sum_{k=1}^{\frac{r+1}{2}} \langle \varepsilon^{2k-1}|Z_\varepsilon(t)|^{2k-1} ,|y_\varepsilon(t)|^{r+1-(2k-1)} \rangle
\\ &\leq -a/2 \norm{y_\varepsilon(t)}^{r+1}_{L^{r+1}} + \underbrace{C}_{\eqqcolon C_{2,y}}\varepsilon^{r+1} \norm{Z_\varepsilon(t)}^{r+1}_{L^{r+1}}.
\end{aligned},$$ where we used Young's product inequality to shift mass onto the term $-a \norm{y_\varepsilon(t)}^{r+1}_{L^{r+1}}$.
Altogether, we can conclude that 
$$\begin{aligned}
& \partial_t \norm{y_\varepsilon(t)}^2_{L^2} \\ &\leq - \kappa_\ast \norm{y_\varepsilon(t)}^2_{\mathcal B^2_{1/2}} + I + II + III \\ &\leq -\kappa_\ast \norm{y_\varepsilon(t)}^{2}_{L^{2}} -\kappa \norm{(-A)^{1/2}y_\varepsilon(t)}^2_{L^2} -a/2 \norm{y_\varepsilon(t)}^{r+1}_{L^{r+1}} \\
& \quad + K_y \sum_{k=3}^{r} \norm{y_\varepsilon(t)}^{k(1-\theta_k)\frac{2}{2-k \theta_k}}_{L^2} + \underbrace{(C_{1,y} + C_{2,y})}_{\eqqcolon C_{y}}\sum_{k=3}^{r+1} \varepsilon^k\norm{Z_\varepsilon(t)}_{L^{k}}^{k}. \\
&\leq -\kappa_\ast \norm{y_\varepsilon(t)}^{2}_{L^{2}} + K_y \sum_{k=3}^{r} \norm{y_\varepsilon(t)}^{p_k}_{L^2} + C_y\sum_{k=3}^{r+1} \varepsilon^k \sup_{0 \leq t \leq T} \norm{Z_\varepsilon(t)}_{\mathcal U}^{k}.
\end{aligned}$$
for $p_k = k(1-\theta_k)\frac{2}{2-k \theta_k} > 2$, where we applied Hölder's inequality to obtain the last line.
\end{proof} 
\begin{proof}[Proof of Lemma \ref{Ybound}] By the preceding lemma, we know that $y_\varepsilon(t) = \norm{y_\varepsilon(t)}^2_{L^2}$ satisfies the differential inequality $$\begin{aligned}
\partial_t y &\leq - \kappa_\ast y + K_y \sum_{k = 3}^{r} y^{p_k/2} + C_y\sum_{k=3}^{r+1} \varepsilon^k\sup_{0 \leq t \leq T} \norm{Z_\varepsilon(t)}_{\mathcal U}^{k},
\end{aligned}$$ for $p_k/2 > 1$ and constants $C_y, K_y > 0$. Now, for $y \in [0,1]$, $p \rightarrow y^p$ is a decreasing function and hence for $p = \max\{p_3/2,...p_r/2\} > 1$ and $0 \leq y \leq 1$, $$\partial_t y \leq -{\kappa_\ast} y + (r-2)K_y y^p +  C_y\sum_{k=3}^{r+1} \varepsilon^k\sup_{0 \leq t \leq T} \norm{Z_\varepsilon(t)}_{\mathcal U}^{k}.$$
To demonstrate the claimed bound, note that for \begin{equation}
y \leq \left(\frac{{\kappa_\ast}}{2(r-2)K_y}\right)^{\frac{1}{p-1}}
\end{equation}
it holds that $$-{\kappa_\ast} y + (r-2)K_yy^p = y((r-2)K_yy^{p-1} - {\kappa_\ast}) \leq -\frac{{\kappa_\ast}}{2}y.$$ Therefore, direct calculation yields that if \begin{equation} \label{Ycond}
\frac{2C_y}{\kappa_\ast} \sum_{k=3}^{r+1} \varepsilon^k\sup_{0 \leq t \leq T} \norm{Z_\varepsilon(t)}_{\mathcal U}^{k} \leq y \leq \left(\frac{{\kappa_\ast}}{2(r-2)K_y}\right)^{\frac{1}{p-1}},
\end{equation} then  $$ -{\kappa_\ast} y + (r-2)K y^p +  C_y\sum_{k=3}^{r+1} \varepsilon^k\sup_{0 \leq t \leq T} \norm{Z_\varepsilon(t)}_{\mathcal U}^{k}  \leq 0.$$ Hence, by standard comparison theorems for first-order ordinary differential equations, it follows that
$$\sup_{t \in [0,T]} y \leq \frac{2C_y}{\kappa_\ast} \sum_{k=3}^{r+1} \varepsilon^k \sup_{0 \leq t \leq T} \norm{Z_\varepsilon(t)}_{\mathcal U}^{k} \vee \underbrace{y_0}_{=0}.$$
Now, it remains to make the dependence on the the supremum of $\sup_{0 \leq t \leq T} \norm{Z_\varepsilon(t)}_{\mathcal U}$ explicit. To this end, consider the function $$\ell(z) \coloneqq \frac{2C_y}{{\kappa_\ast}}(z^3 + \dots + z^{r+1}).$$ Evidently, this map is increasing and hence invertible, so that condition \eqref{Ycond} is equivalent to $$\varepsilon \sup_{0 \leq t \leq T} \norm{Z_\varepsilon(t)}_{\mathcal U} < z_\ast \coloneqq \ell^{-1}\left(\left(\frac{{\kappa_\ast}}{2(r-2)K_y}\right)^{\frac{1}{p-1}}\right),$$ which finalises the proof.
\end{proof}
\begin{corollary} \label{obound}
Let $z_\ast, C_y, \kappa_\ast > 0$ be as in Proposition \ref{Ybound}. Then, if $$\sup_{0 \leq t \leq T} \norm{Z_\varepsilon(t)}_{\mathcal U} \leq \frac{z_\ast}{\varepsilon},$$ 
it follows that
$$\begin{aligned}
\sup_{0 \leq t \leq T} \mathrm{d}(V(t),\Gamma) &\leq  \varepsilon \sup_{0 \leq t \leq T} \norm{Z_\varepsilon(t)}_{L^{2}} + \sqrt{\frac{2C_y}{\kappa_\ast}} \sum_{k=3}^{r+1} \varepsilon^{k/2}  \sup_{0 \leq t \leq T} \norm{Z_\varepsilon(t)}^{k/2}_{\mathcal U}.
\end{aligned}$$
\end{corollary}
\begin{proof}
Since $\mathrm{d}(V(t), \Gamma) \leq \norm{\tilde U(t)}_{L^2} \leq \varepsilon \norm{Z_\varepsilon(t)}_{L^2} + \norm{y_\varepsilon(t)}_{L^2},$ this corollary follows immediately from Proposition \ref{Ybound}.
\end{proof}
\subsection{Growth estimates on the first order approximation}
By virtue of the bounds we derived in Corollary \ref{obound}, it already follows that as $\varepsilon \to 0$, $$ \sup_{0 \leq t \leq T} \norm{\tilde U(t)}_{L^2} \leq  \sup_{0 \leq t \leq T} \varepsilon \norm{Z_\varepsilon(t)}_{L^2} + \sup_{0 \leq t \leq T}\norm{y_\varepsilon(t)}_{L^2} \to 0.$$ However, these upper bounds are quite crude and the dependence on the driver $N$ is unclear. In this section, we make this relationship more explicit by deriving bounds on the supremum of the norm of $Z_\varepsilon(t)$ dependent on the Hölder norm of the path $N$.

\begin{proposition} \label{Zbound}
There exists a constant $C_Z$ dependent on $\norm{\nu\cdot\nabla v^{TW}_0}_{L^2}$, $\norm{f'(v^{TW}_0)}_{L^\infty}$ and $\kappa_\ast > 0$ such that 
\begin{equation}\label{Zboundeq}
\sup_{0 \leq t \leq T} \norm{Z_\varepsilon(t)}_{\mathcal U} \leq  (1+C_Z(1+\lambda)) \sup_{0 \leq t \leq T} \norm{N_{A-\lambda}(t)}_{\mathcal U}
\end{equation}
\end{proposition}
\begin{proof}
We show the bound for the individual terms which constitute $$ \norm{Z_\varepsilon(t)}_{\mathcal U} \coloneqq \norm{Z_\varepsilon(t)}_{L^{2}} \vee \norm{Z_\varepsilon(t)}_{L^{r+1}}.$$ Let $p \in \Set{2,r+1}$. We first apply Bochner's inequality to find that 
$$
\begin{aligned}
&\norm{Z_\varepsilon(t)}_{L^{p}\mathcal O)} \leq  \int_0^t \norm{E(t,s)\left(\tilde P(s) + f'(v^{TW}(s)) + \lambda \right)N_{A-\lambda}(s)}_{L^{p}} \d s + \norm{N_{A-\lambda}(s)}_{L^{p}}.
\end{aligned}
$$
The second term on the right hand side trivially satisfies
$$\begin{aligned}
\norm{N_{A-\lambda}(t)}_{L^{p}} &\leq \sup_{0 \leq t \leq T} \norm{N_{A-\lambda}(t)}_{L^{p}}.
\end{aligned}$$ Now, to find a bound on  $$\int_0^t \norm{E(t,s)\left(\tilde P(s) + f'(v^{TW}(s)) + \lambda \right)N_{A-\lambda}(s)}_{L^{p}} \d s,$$ we can apply the interpolation inequality \eqref{IntPlIneq} to obtain $$\begin{aligned}
\norm{u}_{L^{p}} \leq C \norm{(-A)^{1/2}u}^{\theta_{p}}_{L^2} \norm{u}^{1-\theta_{p}}_{L^2}
\end{aligned}$$ and control the $L^{p}$-norm:
$$
\begin{aligned}
&\int_0^t \norm{E(t,s)\left(\tilde P(s) + f'(v^{TW}(s)) + \lambda \right)N_{A-\lambda}(s)}_{L^{p}} \d s \\&\leq  C \int_0^t \norm{(-A)^{1/2}E(t,s)\left(\tilde P(s) + f'(v^{TW}(s)) + \lambda \right)N_{A-\lambda}(s)}^{\theta_{p}}_{L^2}\\&\quad \quad \quad \times \norm{E(t,s)\left(\tilde P(s) + f'(v^{TW}(s)) + \lambda \right)N_{A-\lambda}(s)}^{1-\theta_{p}}_{L^2} \d s.
\end{aligned}$$
Due to sectoriality of the operator $A$, standard semigroup estimates (\cite{hairer2023introductionstochasticpdes}, Prop. 4.40) yield that there exists a constant $M$ independent of $T$ such that $$\norm{(-A)^{1/2} S(t)x}_{L^2} \leq  \frac{M}{\sqrt{t}}\norm{x}_{L^2}.$$ Combining this with the trivial estimate $$\sup_{0 \leq t \leq T} \norm{\tilde P(s) + f'(v^{TW}(s)) + \lambda}_{L^2 \to L^2} \leq C(1+\lambda)$$ for some constant $C$ independent of $T$, the exponential decay of $E(t,s)$ and Theorem 7.1.3 in \cite{henry81:GTS} imply that $$ \norm{(-A)^{1/2}E(t,s)\left(\tilde P(s) + f'(v^{TW}(s)) + \lambda \right)x}_{L^2} \leq MC(1+\lambda) \frac{e^{-\kappa_\ast (t-s)}}{(t-s)^{1/2}} \norm{x}_{L^2}.$$
Therefore, we arrive at the conclusion that
$$
\begin{aligned}
&\int_0^t \norm{E(t,s)\left(\tilde P(s) + f'(v^{TW}(s)) + \lambda \right)N_{A-\lambda}(s)}_{L^{p}} \d s \\ &\leq CM^{\theta_{p}}(1+\lambda)\int_0^t e^{- \kappa_\ast (t-s)}\frac{\norm{N_{A-\lambda}(s)}_{L^2}}{(t-s)^{\theta_{p}/2}} \d s \\ 
& \leq CM^{\theta_{p}}(1+\lambda)\int_0^\infty \frac{e^{- \kappa_\ast s}}{s^{\theta_{p}/2}} \d s  \sup_{0 \leq t \leq T} \norm{N_{A-\lambda}(t)}_{L^2} \\
&\leq C_Z(1 + \lambda) \sup_{0 \leq t \leq T} \norm{N_{A-\lambda}(t)}_{L^2},
\end{aligned}
$$
for $C_Z \coloneqq CM^{\theta_{p}} \kappa_\ast^{-(1-\theta_p/2)} \Gamma(1-\theta_p/2)$, which is well-defined since $\theta_p \leq 1$. The desired inequality now follows.
\end{proof}
We can improve on the statement of Corollary \ref{obound} using the above Proposition \ref{Zbound} and estimates on the growth of the norm of $N_{A-\lambda}$ in terms of the Hölder norm of the driver $N$ and the time $T$. By distinguishing between the cases $\lambda = 0$ and $\lambda > 0$, we find respective estimates of the right hand side of \eqref{Zboundeq}, in particular $$(1+\lambda) \sup_{0 \leq t \leq T} \norm{N_{A-\lambda}(t)}_{\mathcal U}.$$ This leads to the following improved pathwise stability result, which states that for small times $T$, noise amplitudes $\varepsilon$ and drivers $N$ with small Hölder norm, $\sup_{0 \leq t \leq T} d(V(t), \Gamma)$ remains small.  
\begin{theorem}
\label{ShortLongBounds}Let $z_\ast, C_y, \kappa_\ast > 0$ be as in Proposition \ref{Ybound}. There exist constants $C_S, C_L > 0$ dependent on $\eta$ and $\gamma$ but independent of $\varepsilon, T>0$ such that whenever $$(C_S T^{\eta - \gamma} \wedge C_L) \norm{N}_{C^\eta_0([0,T];\mathcal U_{-\gamma})} \leq \frac{z_\ast}{\varepsilon},$$ 
then both the short time bound
\begin{equation} \label{ShortBound}
\begin{aligned}
\sup_{0 \leq t \leq T} d(V(t), \Gamma) &\leq  C_S T^{\eta-\gamma}\varepsilon \norm{N}_{C^\eta_0([0,T];\mathcal U_{-\gamma})}  + \rho \sum_{k=3}^{r+1} (C_S T^{\eta-\gamma}\varepsilon)^{k/2} \norm{N}_{C^\eta_0([0,T];\mathcal U_{-\gamma})}^{k/2}
\end{aligned}
\end{equation} 
and the long time bound
\begin{equation} \label{LongBound}
\begin{aligned}
\sup_{0 \leq t \leq T} d(V(t),\Gamma) &\leq C_L \varepsilon \norm{N}_{C^\eta_0([0,T];U_{-\gamma})}  +\rho \sum_{k=3}^{r+1}(C_L\varepsilon)^{k/2}\norm{N}^{k/2}_{C^\eta_0([0,T];\mathcal U_{-\gamma})}
\end{aligned}
\end{equation}
hold for $\rho = \sqrt{\frac{2C_y}{\kappa_\ast}}$.
\end{theorem}

\begin{remark}
Tracing through the calculations in the proof of Theorem 1 in \cite{Gubinelli2004YoungIA}, we find that the coefficient $C_S$ in the preceding statement depends on $\eta-\gamma$ with $$C_S \gtrsim \int_0^1 \frac{1}{t^{1-(\eta-\gamma)}} \d t,$$ which means that $C_S \to \infty$ as $\eta \to 0$. Similarly, Lemma \ref{ConvBound} below shows that $C_L$ diverges as $\eta-\gamma \to 0$. However, for fixed $\gamma$, $C_L$ does not diverge as $\eta \to 1$.
\end{remark}

To prove the long-time estimate \eqref{LongBound}, we utilise the following bound on the maximal norm of the convolution and thereby reach an upper bound in terms of a multiple of the Hölder norm of $N$ only. 
\begin{lemma} \label{ConvBound}
Let $0 \leq \gamma < \eta \leq 1$ and $\lambda > 0$. For $p \in \Set{2,r+1}$, the convolution $N_{A-\lambda}$ satisfies $$\begin{aligned}
&\sup_{t \in [0,T]} \norm{N_{A-\lambda}(t)}_{L^p} \leq \lambda^{-(\eta-\gamma)}K(\eta-\gamma) \norm{N}_{C^{\eta}_0([0,T];\mathcal B^p_{-\gamma})}
\end{aligned}$$ for $$K(\eta-\gamma) = \Gamma(\eta-\gamma) + \Gamma(1+\eta-\gamma) + (\eta-\gamma)^{-(\eta-\gamma)}.$$
\end{lemma}
As a direct consequence of Lemma \ref{ConvBound}, we can refine \eqref{Zboundeq} and estimate the norm of $Z_\varepsilon$ in terms of the Hölder norm of the driver $N$.
\begin{corollary}
Let $0 \leq \gamma < \eta \leq 1$. Then \begin{equation} \label{lambgr0bound}
\sup_{0 \leq t \leq T}\norm{Z_\varepsilon(t)}_{\mathcal U} \leq 2C_Z \tilde K(\eta-\gamma)\norm{N}_{C^\eta_0([0,T];\mathcal U_{-\gamma})},
\end{equation}
with $$\tilde K(\eta-\gamma) \coloneqq K(\eta-\gamma) \frac{(\eta-\gamma)^{\eta-\gamma}}{(1-(\eta-\gamma))^{1-(\eta-\gamma)}},$$
where $C_Z$ and $K(\eta-\gamma)$ are the constants given by Proposition \ref{Zbound} and Lemma \ref{ConvBound}, respectively. 
\end{corollary}
\begin{proof}
Given any $\lambda > 0$, we first apply Proposition \ref{Zbound} and then Lemma \ref{ConvBound} to arrive at the upper bound
$$\sup_{0 \leq t \leq T}\norm{Z_\varepsilon(t)}_{\mathcal U} \leq 2C_ZK(\eta-\gamma)\left(\frac{1}{\lambda^{\eta-\gamma}} + \lambda^{1-(\eta-\gamma)} \right)\norm{N}_{C^\eta_0([0,T];\mathcal U_{-\gamma})}.$$
Plugging in $\lambda = \frac{\eta - \gamma}{1-(\eta - \gamma)}$ minimises $$\lambda \mapsto \frac1{\lambda^{\eta-\gamma}} + \lambda^{1-(\eta-\gamma)}$$ and yields the claimed estimate.
\end{proof}

\begin{proof}[Proof of Theorem \ref{ShortLongBounds}] We first show the short time bound \eqref{ShortBound}. We aim to show existence of a constant $C_S$ such that \begin{equation} \label{lamb0}\sup_{0 \leq t \leq T} \norm{Z_\varepsilon(t)}_{\mathcal U} \leq C_S T^{\eta-\gamma}\norm{N}_{C^\eta_0([0,T];\mathcal U_{-\gamma})}.\end{equation} Assuming that \eqref{lamb0} holds, then we would find that $$C_S T^{\eta-\gamma}\norm{N}_{C^\eta_0([0,T];\mathcal U_{-\gamma})} \leq \frac{z_\ast}{\varepsilon}$$ implies $$\sup_{0 \leq t \leq T} \norm{Z_\varepsilon(t)}_{\mathcal U} \leq \frac{z_\ast}{\varepsilon}.$$ Under these conditions, propositions \ref{Ybound} and \ref{Zbound} then imply that $$\begin{aligned}
\sup_{t \in [0,T]}\norm{y_\varepsilon(t)}^2_{L^2} &\leq \frac{2C_y}{\kappa_\ast}\sum_{k=3}^{r+1} \varepsilon^k \sup_{0 \leq t \leq T} \norm{Z_\varepsilon(t)}_{\mathcal U}^{k} \leq \frac{2C_y}{\kappa_\ast}\sum_{k=3}^{r+1} \varepsilon^k (C_S T^{\eta-\gamma}\norm{N}_{C^\eta_0([0,T];\mathcal U_{-\gamma})})^{k}.
\end{aligned}$$
Thus, repeated application of the inequality $\sqrt{x+y} \leq \sqrt{x} +\sqrt{y}$ to $$\begin{aligned}
\sup_{0 \leq t \leq T} \norm{y_\varepsilon(t)}_{L^2} = \sqrt{\sup_{0 \leq t \leq T} \norm{y_\varepsilon(t)}^2_{L^2}} &\leq \sqrt{\frac{2C_y}{\kappa_\ast}\sum_{k=3}^{r+1} \varepsilon^k (C_S T^{\eta-\gamma}\norm{N}_{C^\eta_0([0,T];\mathcal U_{-\gamma})})^{k} } \\
&\leq \sqrt{\frac{2C_y}{\kappa_\ast}} \sum_{k=3}^{r+1} \varepsilon^{k/2} (C_S T^{(\eta-\gamma)}\norm{N}_{C^\eta_0([0,T];\mathcal U_{-\gamma})})^{k/2}
\end{aligned}$$
yields the desired estimate, as \begin{equation*}
\sup_{0 \leq t \leq T} \norm{\tilde U(t)}_{L^2} \leq \varepsilon \sup_{0 \leq t \leq T} 
\norm{Z_\varepsilon(t)}_{L^{2}} + \sup_{0 \leq t \leq T} \norm{y_\varepsilon(t)}_{L^2}.
\end{equation*} 
It remains to show that \eqref{lamb0} holds for some $C_S > 0$. We first apply \eqref{Zboundeq} in the case $\lambda = 0$ and find that 
$$\sup_{0 \leq t \leq T} \norm{Z_\varepsilon(t)}_{\mathcal U} \leq  (1+C_Z)\sup_{0 \leq t \leq T} \norm{N_{A}(t)}_{\mathcal U}.$$
It is left to find a bound on the norm of the process $N_A$ in $L^\infty([0,T];\mathcal U)$.
Since the operator $A$ is not assumed to be dissipative, we cannot exploit any exponential decay property of the semigroup generated by the operators $A_p$. Instead, we apply the the maximal inequality \begin{equation} \label{lamb0Tindel}
\sup_{0 \leq t \leq T} \norm{N_{A_p}(t)}_{\mathcal B^p_\delta} \leq C_2 T^{\eta - \gamma - \delta} \norm{N}_{C^\eta_0([0,T];\mathcal B^p_{-\gamma})}
\end{equation} given by Proposition \ref{YoungConv} for some $C_2 > 0$ independent of $T>0$. For $\delta = 0$, i.e. $\mathcal B^p_\delta = L^p(\mathcal O)$, this gives $$\sup_{0 \leq t \leq T} \norm{N_{A}(t)}_{\mathcal U} \leq 2C_2T^{\eta - \gamma}\norm{N}_{C^\eta_0([0,T];\mathcal U_{-\gamma})}.$$ Therefore, $$\begin{aligned}
\sup_{0 \leq t \leq T} \norm{Z_\varepsilon(t)}_{\mathcal U} \leq  \underbrace{2(1+C_Z)C_2}_{\eqqcolon C_S} T^{\eta - \gamma} \norm{N}_{C^\eta_0([0,T];\mathcal U_{-\gamma})}.
\end{aligned}$$
Using the same approach, but with the upper bound  on $\sup_{0 \leq t \leq T}\norm{Z_\varepsilon(t)}_{\mathcal U}$ given by \eqref{lambgr0bound} instead of \ref{lamb0Tindel}, we obtain the long time bound \eqref{LongBound}.
\end{proof}
\section{Application to self-similar stochastic processes with Hölder continuous paths}
\label{SectionfBM}  
This section is devoted to the study of the effects of $H$-self-similar noise $X^H$ with Hölder continuous paths on the growth of the distance between the orbit $\Gamma$ and the travelling wave solution $V^H$ driven by $X^H$. Let us first introduce the notion of self-similarity.
\begin{definition}[Self-similarity] 
A stochastic process $(X^H_t)_{t \geq 0}$ on some probability space $(\Omega, \mathcal F,\mathbb P)$ with values in a topological vector space $(V,\tau)$ equipped with the induced Borel $\sigma$-algebra is $H$-\textit{self-similar} for $H>0$ if,  for all $a > 0$, $$(X^H_{at})_{t \geq 0}\overset{d}{=}(a^HX^H_t)_{t \geq 0}$$ as laws on the measure space $V^{[0,\infty)}$ equipped with the product sigma-algebra induced by $V$.
\end{definition}
Provided that $V$ is a separable Banach space, Hölder seminorms are equivalent to sequential norms involving dyadic second differences \cite{RackauskasBanach} and then self-similarity of the process $X^H$ implies that for any $T, b \geq 0$, \begin{equation} \label{scaling}
\mathbb P\left(\norm{X^H}_{C^\eta_0([0,T];V)} \leq b \right) = \mathbb P\left(T^{H-\eta}\norm{X^H}_{C^\eta_0([0,1];V)} \leq b \right).
\end{equation}
Therefore, we can control the probability and magnitude of the errors estimates derived in Section \ref{pathstab} by analysing the tail behaviour of the Hölder norm of $X^H$ on the interval $[0,1]$. We now gather a few immediate consequence of self-similarity. Since $$\norm{X(0)}_V \overset{d}{=} \norm{X(a\cdot0)}_V = a^H \norm{X(0)}_V,$$ it immediately follows that $X(0) \equiv 0$. Furthermore, the process $X$ cannot be stationary for $H \neq 0$ as $X(t) \overset{d}{=} a^HX(t)$ implies $a^H = 1$. By similar reasoning, it cannot be that $X_t \overset{d}{=}X_s$ for $X_s,X_t \neq 0$ and $s \neq t$. In principle, $H>0$ can be arbitrarily large, exemplified for example by $$X^{H+n}(t) = \int_0^t \dots \int_0^{t_{n-1}}B^H(t_n)\d t_n \dots \d t_1,$$ which is $(H+n)$-self-similar given a fractional Brownian motion $B^H$. However, the self-similarity of Hölder norms that lies behind the scaling estimate \eqref{scaling} demonstrates that $\eta > H$ is not possible for nondegenerate processes; else we reach a contradiction since $$\mathbb P\left(\norm{X^H}_{C^\eta_0([0,t];V)} >  b\right) \overset{t > s}{\geq} \mathbb P\left(\norm{X^H}_{C^\eta_0([0,s];V)} >  b\right) = \mathbb P\left(\norm{X^H}_{C^\eta_0([0,1];V)} >  s^{\eta-H}b\right) \overset{s \to 0}{\to} 1$$ for all $b,t > 0$. The particular ranges of $\eta$ and $H$ depend both on the marginal distributions of the processes in question and the structure of their increments. We will consider such questions and the derivation of tail estimates of Hölder norms in Appendix \ref{TailEstimateSection}.

\subsection{Stability of travelling waves for self-similar noise with Hölder continuous paths}
Let $(X^H(t))_{t \in [0,T]}$ be a Hölder continuous stochastic process with values in $\mathcal U \coloneqq L^2 \cap L^{r+1}$ on some probability space $(\Omega, \mathcal F, \mathbb P)$. In the following subsection we will see that we can make $H$ large compared to the regularity in time of the driver, and in the context of this subsection, $H>0$ can be arbitarily large.

We consider processes $V^H$ on $[0,T]$ such that $U = V^H-v^{TW}$ solves the equation \begin{equation}\label{fBmNag} \begin{cases} 
\d U(t) = \left(AU(t) + f\left(U(t)+v^{TW}(t)\right)-f\left(v^{TW}(t)\right)\right)\!\d t + \varepsilon \d X^H(t) \\  
U(0) = V(0) -v^{TW}_0 & 
\end{cases}\end{equation} according to Definition \ref{SolDef}. We now state and demonstrate the main theorems of this manuscript. 
\begin{theorem} \label{mainthm}
Suppose that Assumptions \ref{IntPl}, \ref{TW}, \ref{Intgrbl} and \ref{SpcGp} hold and let $A$, $f$ satisfy the assumptions of Theorem \ref{exuniq}, where in particular, $f = f_0 + f_1$ for some odd-order polynomial $f_0$ with $\mathrm{deg}(f_0) = r$ and  $f_1 \in C^2(\mathbb R)$ with bounded first and second derivative.

For $\varepsilon > 0$, let $V^H$ denote the pathwise defined solution of equation \eqref{fBmNag} given by Theorem \ref{exuniq}, with $V^H(0) = v^{TW}_0$.
Then there exist constants $\rho, z_\ast > 0$ independent of $H$, $T$, $\varepsilon$ and $\eta$, and a constant $C_S$ dependent on $\eta$ such that $V^H$ satisfies
\begin{equation} \label{ShortTimeBoundH}
\mathbb P\left(\sup_{0 \leq t \leq T} \mathrm d(V^H(t), \Gamma) \leq C_S  T^H + \rho \sum_{k=3}^{r+1} C_S^{k/2} T^{kH/2} \right) \geq \mathbb P\left(\norm{X^H}_{C^\eta_0([0,1];\mathcal U)} \leq \frac{1}{\varepsilon} \right), ~
\text{ for } T^H \leq \frac{z_\ast}{C_S}.
\end{equation}
\end{theorem}
One striking feature of this short-time estimate is that higher Hurst indices $H$ yield better error estimates, independent of the maximal Hölder exponent $$\eta_X \coloneqq \sup \{ \eta < H: X^H \in C^\eta_0([0,T];\mathcal U)\}.$$ In particular, for fixed $\varepsilon$, the probability that this estimate holds is bounded below uniformly as $T \to 0$. This result is in particular valid for processes with high self-similarity and low Hölder exponent.
\begin{proof}[Proof of Theorem \ref{mainthm}] Let $X^H$ denote a fixed realisation a path of driving process, and let $V^H$ be the corresponding solution of equation \eqref{fBmNag} with initial condition $V^H(0) = v^{TW}_0$. By Proposition \ref{ShortLongBounds}, we know that there exist a threshold $z_\ast>0$ and a constant $C_S > 0$ independent of $\eta, \varepsilon > 0$ such that whenever
\begin{equation} \label{ShortCondition}
C_S T^{\eta} \norm{X^H}_{C^\eta_0([0,T];\mathcal U)} \leq \frac{z_\ast}{\varepsilon}
\end{equation}
holds, then 
$$
\begin{aligned}
\sup_{0 \leq t \leq T} d(V^H(t), \Gamma) &\leq  C_S T^{\eta}\varepsilon \norm{X^H}_{C^\eta_0([0,T];\mathcal U)} + \rho \sum_{k=3}^{r+1} (C_S T^{\eta}\varepsilon)^{k/2} \norm{X^H}_{C^\eta_0([0,T];\mathcal U)}^{k/2}.
\end{aligned}$$
for $\rho = \sqrt{\frac{2C_y}{\kappa_\ast}}$ with constants $\kappa_\ast, C_y$ as introduced in Assumption \ref{SpcGp} and  Proposition \ref{diffineq}. Close inspection of the involved terms shows that if $$C_S T^{\eta} \varepsilon \norm{X^H}_{C^\eta_0([0,T];\mathcal U)} \leq z_\ast \wedge C_S T^H,$$ then both \eqref{ShortCondition} and $$\sup_{0 \leq t \leq T} \mathrm d(V^H(t), \Gamma) \leq C_S  T^H + \rho \sum_{k=3}^{r+1} C_S^{\frac{k}{2}} T^{\frac{kH}{2}}$$ follow. Therefore, 
$$
\begin{aligned}
\mathbb P\left(\sup_{0 \leq t \leq T} \mathrm d(V^H(t), \Gamma) \leq C_S  T^H + \rho \sum_{k=3}^{r+1} C_S^{k/2} T^{kH/2} \right) \geq \mathbb P\left(\varepsilon C_S T^{\eta} \norm{X^H}_{C^\eta_0([0,T];\mathcal U)} \leq z_\ast \wedge C_S T^H\right).
\end{aligned}$$
Now, crucially, Hölder norms of self-similar processes satisfy the scaling identity $$T^{\eta} \norm{X^H}_{C^\eta_0([0,T];\mathcal U)} \overset{d}{=} T^H \norm{X^H}_{C^\eta_0([0,1];\mathcal U)}.$$ We can conclude that \begin{equation*}
\mathbb P\left(\varepsilon C_S T^{\eta} \norm{X^H}_{C^\eta_0([0,T];\mathcal U)} \leq z_\ast \wedge C_S T^H\right) = \mathbb P\left(\norm{X^H}_{C^\eta_0([0,1];\mathcal U)} \leq \frac{1}{\varepsilon} \left(\frac{z_\ast}{C_S T^H} \wedge 1 \right)\right).\qedhere 
\end{equation*} \end{proof}
If one applies \eqref{LongBound} instead of \eqref{ShortBound}, one obtains a long-time bound in probability. This estimate is relevant for large time-scales only, where $T^\eta \gg 1$ implies that \eqref{lambgr0bound} provides a significantly lower upper bound of $\norm{Z_\varepsilon}_{L^\infty([0,T];\mathcal U)}$. 
\begin{theorem}\label{LongboundSelfSim}
Under the same conditions as Theorem \ref{mainthm}, there exist constants $\rho, z_\ast, C_L > 0$ with 
\begin{equation}\label{LongTimeBoundH}
\begin{aligned}
& \mathbb P\left(\sup_{0 \leq t \leq T} \mathrm d(V^H(t), \Gamma) \leq C_L\delta  + \rho \sum_{k=3}^{r+1} C_L^{k/2}\delta^{k/2} \right) \geq \mathbb P\left(\norm{X^H}_{C^\eta_0([0,1];\mathcal U)} \leq \frac{\delta}{\varepsilon T^{H-\eta}}\right), ~\text{ for } \delta \leq \frac{z_\ast}{C_L}.
\end{aligned}
\end{equation}
\end{theorem}
We now apply this theorem to self-similar processes with tail behaviour corresponding to the examples in the appendix (cf. Appendix \ref{TailEstimateSection}). We specifically focus on the cases where either \begin{equation} \label{PolDecayIII}
\mathbb P\left(\norm{X^H}_{C^\eta_0([0,1];\mathcal U)} > b\right) \in \mathcal O\left(b^{-\alpha}\right), ~ b \to \infty \tag{I}
\end{equation} or 
\begin{equation}\label{ExpDecayIII}
\mathbb P\left(\norm{X^H}_{C^\eta_0([0,1];\mathcal U)} > b\right) \in \mathcal O(\exp(-b^\alpha/k)), ~ b \to \infty \tag{II}
\end{equation}
for some $\alpha, k > 0$. For self-similar processes with stationary increments, this tail behaviour corresponds to the tail of the marginal $X^H(1)$ (see Lemma \ref{StationaryTail} in the appendix).
\begin{corollary}\label{AppliedSelfSim}
Assume \eqref{PolDecayIII} or \eqref{ExpDecayIII} holds. Then, as $\varepsilon \to 0$, $V^H$ satisfies the short term bound
\begin{equation} \label{ShortTimeBoundAppl}
\begin{aligned}
&\mathbb P\left(\sup_{0 \leq t \leq T} \mathrm d(V^H(t), \Gamma) \leq C_S  T^H + \rho \sum_{k=3}^{r+1} C_S^{k/2} T^{kH/2} \right) \geq \begin{cases}
1- \mathcal O\left(\varepsilon^{\alpha}\right) & \eqref{PolDecayIII} \\
1- \mathcal O\left(\exp(-\varepsilon^{-\alpha}/k\right) & \eqref{ExpDecayIII}
\end{cases}
\end{aligned}
\end{equation}
for small $T$. Further, as $T\to \infty$, the scaling $\varepsilon=T^{-(H-\eta)}$ yields the long time bound
\begin{equation} \label{LongTimeBoundAppl}
\begin{aligned}
& \mathbb P\left(\sup_{0 \leq t \leq T} \mathrm d(V^H(t), \Gamma) \leq C_L\delta + \rho \sum_{k=3}^{r+1} C_L^{k/2}\delta^{k/2}\right) \geq \begin{cases}
1- \mathcal O\left(\delta^{-\alpha} T^{-\alpha\beta}\right) & \eqref{PolDecayIII} \\
1- \mathcal O\left(\exp(-\delta^{\alpha} T^{\alpha\beta})\right) & \eqref{ExpDecayIII}
\end{cases}
\end{aligned}
\end{equation}
for any $\eta + \beta < \eta_X$ and small $\delta$.
\end{corollary}
\begin{remark}
In the case of fBm, $\eta_X = H$, while fractional $\alpha$-stable processes (cf. Section \ref{StationaryTailSection}) with $1 < \alpha < 2$ satisfy the relation $\eta_X = H-1/\alpha$.
\end{remark}
\begin{remark}
Observe that we can put $\delta = T^{-\gamma}$ for any $\gamma < \beta$ and still see that the right-hand side of \eqref{LongTimeBoundAppl} tends to $1$ as $T \to \infty$. Therefore, with this scaling, the error vanishes in probability.
\end{remark}

It would be desirable to know inhowfar these estimates are optimal. Unfortunately, even in one dimension, maximal inequalities for non-Markovian processes are difficult to obtain. However, some heuristic arguments suggest that our results are rather accurate polynomial bounds for $\alpha$-stable processes. We compare $Z_\varepsilon$ (cf. Corollary \ref{Zbound})\ to some one-dimensional Ornstein-Uhlenbeck process $$Z(t) \coloneqq\int_0^t e^{-(t-s)} \d Y^H(s)$$ with $Y^H$ self-similar and $\alpha$-stable for $1 < \alpha \leq 2$. For small times $t>s>0$, $e^{-(t-s)}\approx 1$ so that $Z(t) \approx Y^H(t)$. Therefore, the running maximum of $Z(t)$ approximately behaves as $T^H$ on short times, which is in line with our results. This calculation is up to a random constant; in our results, the necessary scaling of the noise amplitude to counterbalance this constant is determined by the tails of marginals of $Y^H$. 

Now, we want to justify that the scaling $\varepsilon \sim T^{-(H-\eta)}$ is natural and leads to decay as time progresses. First consider the case where $Y^H(t)$ is a fractional Brownian motion. Then $H-\eta_Y = 0$ and since $\sup_{t \geq 0}\mathbb E[Z^2_t] <\infty$, Dudley's metric entropy bound (\cite{TalagrandProbabilityBanach}, Ch. 12) implies that $$\E\left[\sup_{0\leq t\leq T} |Z(t)|\right] \lesssim (\log 1+T)^{1/2}$$ grows subpolynomially. Hence, $ \varepsilon \E\left[\sup_{0\leq t\leq T} |Z(t)|\right] \lesssim T^{-\beta} (\log 1+T)^{1/2} \overset{T\to\infty}{\to} 0$ for any $\beta = H-\eta >0$.  On the other hand, if $\alpha \in (1,2)$, then the relation $H-\eta_Y = \frac{1}{\alpha}$ holds. If we make the strong assumption that the damping factor decorrelates subsequent blocks of the Ornstein-Uhlenbeck process, we can treat $\sup_{0\leq t\leq T} |Z(t)|$ as the maximum of, say, $\lfloor T\rfloor$ independent random variables. Asymptotically, the expected maximum grows as $T^{\frac1\alpha}$ for $\alpha \in (1,2)$ (compare to corresponding continuous-time processes). In the first case, this gives us the same type of decay as above. This implies that $$\varepsilon \mathbb E\left[\sup_{0\leq t\leq T} |Z(t)|\right] \lesssim T^{-\frac{1}{\alpha}-\beta} T^{\frac1\alpha} \overset{T\to\infty}{\to} 0$$ for any $\beta >0$.  Hence we observe significantly larger errors. We see that for fractional $\alpha$-stable processes, this scaling naturally leads to decay of the error over large timespans and demonstrates a positive relation between heavier tails of marginals and increased error over long timespans. We interpret these scaling relationships as an indication that the exponential timescales derived in \cite{HamsterCHS2020SoTW} also apply to the case of fractional Brownian noise, while stable noise seems to yield at most quadratic timescales.
\appendix
\section*{Appendix}
\section{Identities for convolutions against Hölder paths}
\begin{proof}[Proof of Proposition \ref{ibp}]
As in \citet{Gubinelli2004YoungIA}, we consider $$\mathfrak S^n(t) = \sum_{k=0}^{2^n-1} S(t-t^n_k)(N(t^n_{k+1})-N(t^n_k)),$$ where $t^n_k = \frac{2^n t}{k}$. In \cite{Gubinelli2004YoungIA}, it was shown that $\mathfrak S^n(t)$ converges in $\mathcal B_\delta$ for each $t \in [0,T]$, and that the resulting process is Hölder continuous. To prove the desired identity, it suffices to show that $\mathfrak S^n$ converges pointwise to the expression on the right hand side of  \eqref{ibpformula}. By rearranging the Riemann sum, we find that 
$$\begin{aligned}
\mathfrak S^n(t) &= S(1/2^n)N(t) - \sum_{k=0}^{2^n-2} (S(t-t^n_{k+1})-S(t-t^n_k))N(t^n_{k+1})\\
&= S(t)N(t) - \sum_{k=0}^{2^n-2} (S(t-t^n_{k+1})-S(t-t^n_k))(N(t^n_{k+1})-N(t)).
\end{aligned}$$
Since, for $s < s' < t$ and $x \in \mathcal B$, $$S(t-s')x - S(t-s)x = - \int_s^{s'} A S(t-r) x \d r,$$ it follows that
$$
\begin{aligned}
\mathfrak S^n(t) = S(t)N(t) + \int_0^t AS(t-r) (\Delta_n N)(r) \d r,
\end{aligned}$$
where $$(\Delta_n N)(r) = \sum_{k=0}^{2^n-2} \mathds{1}_{[t^n_k, t^n_{k+1})}(r)(N(t^n_{k+1})-N(t)).$$
To show convergence in $\mathcal B_\delta$, it remains to prove that $$\int_0^t (-A)^\delta A S(t-r)(\Delta_n N)(r) \d r \to \int_0^t (-A)^\delta A S(t-r) (N(r)-N(t)) \d r.$$ By Hölder continuity of $N$ and closedness of $A$, we see that for $r < t$, $$ A S(t-r)(\Delta_n N)(r) \to A S(t-r) (N(r)-N(t)).$$ To prove the claim using the dominated convergence theorem for Bochner integrals, it remains to show existence of an integrable majorant. To this end, first note that $(\Delta_n N)(r) \equiv 0$ on $[t-1/2^n,t]$. Henceforth, fix $r < t-1/2^n$. It then holds that 
$$\begin{aligned}
\norm{(-A)^\delta AS(t-r)(\Delta_n N)(r)}_{\mathcal B_\delta} &\leq M \frac{1}{(t-r)^{1+\delta + \gamma}} \norm{N(\lceil 2^n r \rceil /2^n) - N(t)}_{\mathcal B_{-\gamma}}\\
& \leq M \frac{(t-\lceil 2^n r \rceil /2^n)^\eta}{(t-r)^{1+\delta+\gamma}} \norm{N}_{C^\eta_0([0,T];\mathcal B_{-\gamma})}.
\end{aligned}$$
Now, since $r \leq \lceil 2^n r \rceil /2^n$, we find that $t - \lceil 2^n r \rceil /2^n \leq t - r$ and thus $$\frac{(t-\lceil 2^n r \rceil /2^n)^\eta}{(t-r)^{1+\delta+\gamma}} \leq \frac{1}{(t-r)^{1-(\eta - \delta - \gamma)}}.$$
As $1-(\eta - \delta - \gamma) < 1$, $$\norm{AS(t-\cdot)(\Delta N)^n(\cdot)}_{\mathcal B_\delta} \leq M\frac{\norm{N}_{C^\eta_0([0,T];\mathcal B_{-\gamma})}}{(t-\cdot)^{1-(\eta-\delta-\gamma)}} \in L^1([0,t])$$ uniformly in $n$, which finishes the proof.
\end{proof}
\begin{proof}[Proof of Proposition \ref{DampDuham}]
The proof again relies on the integration by parts formula $$N_{A-\lambda}(t) = \int_0^t (A-\lambda)e^{-\lambda(t-s)}S(t-s)\left(N(s)-N(t)\right)\d s + e^{-\lambda t}S(t)N(t).$$ Analyticity of the semigroup $e^{-\lambda t}S(t)$ now implies 
$$
\begin{aligned}
\norm{N_{A-\lambda}(t))}_{\mathcal B^p_{-\gamma}} &\leq \int_0^t \norm{(A-\lambda)e^{-\lambda(t-s)}S(t-s)\left(N(s)-N(t)\right)}_{L^p(\mathcal O)} \d s \\& \quad + \norm{e^{-\lambda t} S(t)N(t)}_{L^p(\mathcal O)}
\\ &\leq \int_0^t  e^{-\lambda(t-s)}\left((t-s)^{-(1+\gamma)} + \lambda (t-s)^{-\gamma}\right) \norm{N(s)-N(t)}_{\mathcal B^p_{-\gamma}} \d s \\ & \quad + t^{-\gamma}e^{-\lambda t} \norm{N(t)}_{\mathcal B^p_{-\gamma}}
\end{aligned}
$$
Note that as $N \in C^\eta_0([0,T];\mathcal B^p_{-\gamma})$, $$\norm{N(s)-N(t)}_{\mathcal B^p_{-\gamma}} \leq \norm{N}_{C^\eta_0([s,t];\mathcal B^p_{-\gamma})} (t-s)^\eta$$ for $0 < \eta < H$ and $s,t \in [0,T]$. Consequently,
$$
\begin{aligned}
\norm{N_{A-\lambda}(t)}_{\mathcal B^p_{-\gamma}} &\leq \int_0^t  e^{-\lambda(t-s)}\left((t-s)^{-1+ (\eta - \gamma)}  + \lambda(t-s)^{\eta-\gamma} \right)  \norm{N}_{C^\eta_0([s,t];\mathcal B^p_{-\gamma})} \d s \\ & \quad + e^{-\lambda t} t^{\eta - \gamma} \norm{N}_{C^\eta_0([0,t];\mathcal B^p_{-\gamma})}
\\ &\leq \Bigg(  \int_0^t e^{-\lambda(t-s)}(t-s)^{-1+ (\eta - \gamma)} \d s + \int_0^t \lambda e^{-\lambda(t-s)}(t-s)^{\eta-\gamma} \d s + e^{-\lambda t} t^{\eta-\gamma} \Bigg) \\ &  \quad  \times \norm{N}_{C^\eta_0([0,t];\mathcal B^p_{-\gamma})}
\\ &\leq \lambda^{-(\eta-\gamma)} \left(\Gamma(\eta-\gamma) + \Gamma(1+\eta-\gamma) + (\eta-\gamma)^{-(\eta-\gamma)}\right)\norm{N}_{C^\eta_0([0,t];\mathcal B^p_{-\gamma})},
\end{aligned}
$$
and the claim follows.
\end{proof}

\section{Tail estimates on Hölder norms of self-similar processes} \label{TailEstimateSection}
The purpose of this section is to show that important classes of self-similar processes $X^H$ satisfy either one of the bounds \eqref{PolDecayIII} and \eqref{ExpDecayIII}. Our motivating examples are the linear stable fractional Lévy motions $$Z^H(t) = \int^t_{-\infty}(t-s)^{H-1/\alpha}-(-s)^{H-1/\alpha}_+ \d L^\alpha(s),$$ which in particular include fractional Brownian motion, driven by some two-sided $\alpha$-stable Lévy process $L^\alpha$ with values in a Banach space $\mathcal B$ equipped with the Borel sigma algebra. Such processes have been considered as drivers of SPDE for example in \cite{FracLevyProcBanach}. Recall that an $\alpha$-stable Lévy process $(L^\alpha_t)_{t \geq 0}$ is defined as a stochastically continuous, $\mathcal B$-valued, almost surely càdlàg process, adapted on some filtered probability space $(\Omega, \mathcal F, (\mathcal F_t)_{t \geq 0}, \mathbb P)$ with $L^\alpha_0 \equiv 0$ that exhibits independent, stationary $\alpha$-stable increments.

Following \citet{AcostaStable}, we define a stable measure $\mu$ on the Borel sets of $\mathcal B$ through the property that for any $a, b > 0$, there exist $c> 0$ and $d \in \mathcal B$ with $$aX + bY \overset{d}{=} c X + d$$ whenever $X, Y \sim \mu$ are independent. As noted in the exposition there, if $\mu$ is symmetric (i.e. $\mu(A) = \mu(-A)$ for any Borel set $A$) then $d = 0$ and further, there exists $0 < \alpha \leq 2$ such that $c = (a^\alpha + b^\alpha)^{1/\alpha}$. Stable measures can also be characterised by the distributions of evaluations by functionals, since $\mu$ is symmetric $\alpha$-stable iff $\mu \circ \ell^{-1}$ is symmetric $\alpha$-stable on $\mathbb R$ for any linear functional $\ell \in F \subset \mathcal B^\ast$, where $(F,\mathcal B)$ form a semifull pair. An equivalent spectral formulation of this result is that there exists a function $\psi \colon F \rightarrow \mathbb R$ such that $\hat \mu(\ell) = \exp(-\psi(\ell))$ for any $\ell \in F$. The function $\psi$ is negative-definite and sequentially $\sigma(F, \mathcal B)$-continuous with $\psi(0) = 0, \psi \geq 0$ and $\psi(t\ell) = |t|^\alpha\psi(\ell)$. A thorough study of (symmetric) $\alpha$-stable distributions in Banach spaces can be found in \cite{LindeStable} and conditions for the existence of corresponding Lévy processes, along with sample path properties and conditions for sensible theories of stochastic integration can be found for example in \cite{DettweilerLevyProcBanach} or \cite{RiedleLevy}. We shall not be concerned here about the geometry of the Banach space since the processes relevant to us take values in the separable, reflexive, type $2$ Banach space $\mathcal B = L^2 \cap L^{r+1}$. 

We only sketch the construction of the integrals $Z^H$, as it follows the approach described in \cite{DettweilerLevyProcBanach}, \cite{RosinskiStochIntStable} or \cite{FracLevyProcBanach}. For each $t$, approximating the integrands with step functions, we can exploit independence of increments, stationarity, self-similarity and stability to obtain a sequence of random variables $$\begin{aligned}
X_n = \sum_{1 \leq k \leq n} c_{n,k}(L^\alpha_{t_{n+1}}-L^\alpha_{t_n}) \overset{d}{=} \sum_{1 \leq k \leq n} c_{n,k}L^{n,\alpha}_{t_{n+1}-t_n} &\overset{d}{=} \sum_{1 \leq k \leq n} c_{n,k}(t_{n+1}-t_n)^{1/\alpha}L^{n,\alpha}_1 \\& \overset{d}{=} \left(\sum_{1 \leq k \leq n} c^\alpha_{n,k}(t_{n+1}-t_n)\right)^{1/\alpha} L_1
\end{aligned}$$
which approximates the integrals. Here, $c_{n,k}$ are suitable point evaluations stemming from Riemann-Stieltjes sums, and $L^{n,\alpha} \sim L^\alpha$ are independent copies. By the mentioned properties of characteristic distributions of stable variables, these variables converge in distribution to $Z^H$ provided that $$\int_{-\infty}^t\left((t-s)^{H -1/\alpha} -(-s)_+^{H-1/\alpha}\right)^\alpha\d s < \infty.$$ This is given (cf. Thm 5.7 in \cite{FracLevyProcBanach}) for $1/\alpha < H < 1$, which imposes that $\alpha > 1$. As in the subsequent Theorems 5.8 and 5.9 therein, similar spectral arguments show self-similarity and stationarity of increments of the process associated with the measure constructed on $\mathcal B^{[0,\infty)}$ by the Kolmogorov extension theorem. We will study continuity and tails of Hölder norms of $Z^H$ in Section \ref{StationaryTailSection}.

Another natural way to construct an $H$-self-similar process (optionally with stationary increments) given i.i.d. copies $(X^H_i)_{i \geq 1}$ of a one-dimensional $H$-self-similar process is to define $$X^H = \sum_{i \geq 1} \lambda_i X^H_i e_i$$ given a Schauder Basis $(e_i)_{i \geq 1}$ of $\mathcal B$ and a sequence $\lambda_i > 0$ with suitable decay. Then $$\norm{X^H}_{C^\eta_0([0,T];\mathcal B)} \coloneqq \norm{\sum_{i \geq 1} \lambda_i X^H_i x_i}_{C^\eta_0([0,T];\mathcal B)} \leq \sum_{i \geq 1} \lambda_i \norm{x_i}_{\mathcal B} \norm{X^H_i}_{C^\eta_0([0,T];\mathbb R)}$$ satisfies the same tail estimates as the Hölder norm of $X_i$. The following Lemma provides some sufficient conditions for this to be the case for more general sums of i.i.d. variables, so that the above method also applies to other Banach spaces of paths, such as $L^\infty([0,T];\mathcal B)$. Depending on the geometry of $\mathcal B$ and the distribution of $X^H_i$, these conditions are far from necessary or sharp, but they are easily applicable in a wide range of cases. For simplicity, we assume that the processes in question are identically distributed and have identical tail decay, but note that this is not strictly necessary.

\begin{proposition} \label{SumTailEstimate}
Let $X_i \sim X$ be a sequence of i.i.d. non-negative real-valued random variables and $(\lambda_i)_{i \geq 1} \subset \mathbb R$ be a sequence with $\mathbb P\left(\sum_{i \geq 1} \lambda_iX_i < \infty \right) = 1$
\begin{enumerate}
    \item First, suppose that there exists $0 < \alpha < 1$ with $\mathbb P(X > t) \in \mathcal O(t^{-\alpha})$ and $\sum_{i \geq 1} \lambda_i^\alpha < \infty$. Then $\mathbb P\left( \sum_{i\geq 1} \lambda_i X_i > t\right) \in \mathcal O(\sum_{i \geq 1} \lambda_i^\alpha t^{-\alpha})$.
\end{enumerate}
Now assume additionally that there exists some $0 
< \beta < 1$ with $\sum_{i \geq 1} \lambda^{1-\beta}_i, \sum_{i \geq 1} \lambda^{\beta \alpha}_i < \infty.$ Then, for any $\alpha > 0$, 
\begin{enumerate}
    \item[2.] $\mathbb P(X > t) \in \mathcal O(t^{-\alpha})$ implies that $\mathbb P\left( \sum_{i\geq 1} \lambda_i X_i > t\right) \in \mathcal O(C_{\alpha,\beta}t^{-\alpha})$
    \item[3.] $\mathbb P(X > t) \in \mathcal O(\exp(-t^\alpha/k)$ implies $\mathbb P\left( \sum_{i\geq 1} \lambda_i X_i > t\right) \in \mathcal \mathcal O\left(\exp\left(-C_{\alpha,\beta}^{-1}t^\alpha/k\right)\right).$
\end{enumerate}
for $C_{\alpha, \beta} = \sum_{i\geq 1} \lambda_i^{\alpha\beta}\left(\sum_{i\geq 1} \lambda_i^{1-\beta}\right)^\alpha$.
\end{proposition}

\begin{proof}[Proof of Lemma \ref{SumTailEstimate}]
First, assume that $\alpha < 1$. Let $\hat\phi_X$ denote the Laplace transform of $X$ and $\hat \phi_\Sigma$ denote the Laplace transform of $\sum_{i \geq 1} \lambda_i X_i$. By assumption, $$\begin{aligned}
1-\hat \phi_X(s) = \mathbb E[1-e^{-sX}] &= \int_0^\infty \mathbb E[se^{-sx}\mathds{1}_{[0,X)}(t)] \d t \\&= \int_0^\infty s e^{-st}\mathbb P(X > t) \d t \\ & \leq C \int_{0} se^{-st} t^{-\alpha}\d t \\&= Cs^\alpha  \int_0^\infty e^{-t}t^{-\alpha} \d t
\end{aligned}$$ so $\hat \phi_X(s) \geq 1- C_\alpha s^\alpha$ for small $s$ and $C_\alpha \coloneqq C \Gamma(1-\alpha)$. Our goal is to now translate this estimate into an estimate on $\hat \phi_\Sigma$. The Lebesgue DCT implies that $$ \hat \phi_\Sigma(s) = \mathbb E \left[e^{-s\sum_{i \geq 1} \lambda_i X_i}\right] = \prod_{i \geq 1} \mathbb E \left[e^{- \lambda_i s X_i}\right] =  \prod_{i \geq 1} \hat \phi_X(\lambda_i s).$$ Thus, for small enough $s$, we find that $\hat \phi_\Sigma(s) \geq \prod_{i \geq 1} 1- C_\alpha \lambda^{\alpha}_i s^{\alpha}.$ The elementary estimate $-x \geq \log(1-x) \geq -\frac{x}{\sqrt{1-x}}$
then first implies that $ \hat \phi_\Sigma(s) \geq e^{-\tilde C'_\alpha \sum_{i \geq 1} \lambda^{\alpha}_i s^{\alpha}}$ for some constant $C'_\alpha$ and we subsequently that $$\hat \phi_\Sigma (s) \geq 1-C'_\alpha \sum_{i \geq 1} \lambda^{\alpha}_i s^{\alpha}$$ for $s \ll 1$. Now, as a consequence, we find that for large $t$, $$\begin{aligned}
C'_\alpha \sum_{i \geq 1} \lambda^\alpha_i \cdot t^{-\alpha} \geq 1-\hat \phi(1/t) = \mathbb E[1-e^{-X/t}] \geq  (1-e^{-1})\mathbb E[\mathds 1_{\Set{X > t}}] = (1-e^{-1})\mathbb P(X > t).
\end{aligned}$$
To prove the second part of the proposition, assume w.l.o.g. that $\sup_i \lambda_i \leq 1$. We first remark that by assumption, there exists some $C' >0$ such that for any $t \gg 1$, $$\mathbb P\left( \forall i \colon X_i \leq \frac{t}{\lambda^\beta_i}\right) = \prod_{i \geq 1} \mathbb P\left( X_i \leq \frac{t}{\lambda^\beta_i}\right) \geq \prod_{i \geq 1} 1 - C' \lambda_i^{\alpha \beta} t^{-\alpha} \geq e^{- \tilde C\sum_{i \geq 1} \lambda_i^{\alpha \beta} \cdot t^{-\alpha}} \geq 1-\tilde C\sum_{i \geq 1} \lambda_i^{\alpha \beta} \cdot t^{-\alpha},$$ where the constant $C$ derives from the logarithmic estimate used in the proof of the first statement. As $$\mathbb P\left(\sum_{i \geq 1} \lambda_i X_i \leq \sum_{i \geq 1} \lambda^{1-\beta}_i t\right) \geq \mathbb P\left( \forall i \colon X_i \leq \frac{t}{\lambda^\beta_i}\right),$$ the proposed estimates follow in the first case. To prove the second case, we repeat the above calculation and apply the fact that $$\sum_{i \geq 1} \exp{\left(-\frac{b^\alpha}{k}\left(\frac{1}{\lambda^{\alpha \beta}_i} - \frac{1}{\sum_{n \geq 1} \lambda^{\alpha \beta}_n}\right)\right)}$$ is bounded to derive the desired asymptotic.
\end{proof}
\subsection{Self-similar processes with stationary increments} \label{StationaryTailSection}
Suppose henceforth that $X^H$ possesses stationary increments. Furthermore, suppose that either \begin{equation} \label{PolDecay}
\mathbb P\left(\norm{X^H(1)}_\mathcal B > b\right) \in \mathcal O\left(b^{-\alpha}\right), ~ b \to \infty \tag{I}
\end{equation} or
\begin{equation}\label{ExpDecay}
\mathbb P\left(\norm{X^H(1)}_\mathcal B > b\right) \in \mathcal O\left(\exp(-b^\alpha/k)\right), ~ b \to \infty \tag{II}
\end{equation}
for some $\alpha, k > 0$. Various examples of one-dimensional processes that satisfy these assumptions are described in (\cite{TaqquSamorodnitsky}, Ch. 7) or \cite{EmbrechtsMaejimaSelfSimilar}. Infinite-dimensional extensions, as well as stochastic differential equations driven by non-Gaussian self-similar processes in Banach spaces have been investigated for example in \cite{FracLevyProcBanach} or \cite{MaslowskiOndrejat}. 

Define the Dudley metric as $d_{X,r}(t,s)= \mathbb E\left(\lVert X_t-X_s\rVert^r_{\mathcal B}\right)^{1/r}$ on $[0,T]$, for $ 0 < r < \alpha$. It is well known that the topological entropy with respect to this metric contains important information on the regularity of the stochastic process $X^H$ (\cite{TalagrandProbabilityBanach} Ch. 11 and 12). A less intricate avenue towards estimates on Hölder norms in particular is open in the case that $X$ has stationary increments. Then, as a consequence of stationarity of increments and self-similarity, $d_{X,r}(t,s) \propto |t-s|^{H}.$ This in particular implies that $\mathbb E[|X_t-X_s|^r] \propto |t-s|^{rH}$ and we can apply the Kolmogorov continuity theorem for complete metric spaces. 

In case \eqref{ExpDecay}, we then find that there exists a version $Y$ of $X$ such that $Y \in C^\eta_0([0,1];\mathcal B)$ for all $0 <\eta < H$ almost surely. In case \eqref{PolDecay}, we need to assume that $H > 1/\alpha$ and closer inspection shows that $Y \in C^\eta_0([0,1];\mathcal B)$ almost surely for all $0 <\eta < H - 1/\alpha$.
\begin{remark}
In case \ref{PolDecay}, the condition $H > 1/\alpha$ excludes the possibility that $\alpha \leq 1$ and it turns out that for general $\alpha$-stable processes, this condition is necessary for $\alpha \neq 2$: some standard self-similar, $\alpha$-stable processes are Hölder continuous only for $1 < \alpha < 2$ and $H \in (1/\alpha, 1]$ (see Examples 5 to 7 in  \cite{TaqquSamorodnitsky}, Ch. 10.2). For $\alpha = 2$, it is of course well-known that the unique $H$-self similar Gaussian process with stationary increments, the fBm, is $\eta$-Hölder continuous for all $0 < \eta < H$. 
\end{remark} 
\begin{lemma} \label{StationaryTail}
Let $X^H$ be $H$-self-similar with stationary increments.
\begin{enumerate}
    \item Assume that \eqref{PolDecay} holds for some $\alpha > 1$ and $H > 1/\alpha$. Then   for any $0 < \eta < H-1/\alpha$. $$ \mathbb P\left( \sup_{0 \leq s < t \leq 1} \frac{\norm{ X(t)-X(s)}_\mathcal{B}}{\lvert t-s\rvert^\eta}  > b\right) \in \mathcal O\left(b^{-\alpha}\right),~b \to \infty.$$
    \item Assume that \eqref{ExpDecay} holds for some $\alpha,k > 0$. Then, for any $0 < \eta < H$ $$ \mathbb P\left( \sup_{0 \leq s < t \leq 1} \frac{\norm{ X(t)-X(s)}_\mathcal{B}}{\lvert t-s\rvert^\eta}  > b\right) \in \mathcal O\left(\exp(-K^{-\alpha}_\eta b^{\alpha}/k)\right),~b \to \infty$$
    with $K_\eta = \frac{4}{(2^\eta-1)(2^{1-\eta}-1)}$.
\end{enumerate}
\end{lemma}
As a consequence of Proposition \ref{StationaryTail}, one can obtain tail estimates on Hölder norms of fractional $\alpha$-stable Levy processes such as fractional Brownian motion or the linear stable fractional motion, and other non-Gaussian processes such as the Rosenblatt process, provided we know tail estimates of the marginal at time $t = 1$. It is further of note that this result yields tails for Hölder norms of iterated processes such as $X^{H_1}(\beta^{H_2}(t))$, where $X^{H_1}$ is a two-sided $\mathcal B$-valued Hölder continuous self-similar process and $\beta^{H_2}(t) \in C^\eta_0([0,\infty))$ is real-valued and self-similar.

For the proof of Proposition \ref{StationaryTail}, we need a generalisation of the linear homeomorphism from the Hölder space $C_0^{\eta}([0,1];\mathbb R)$  to the space $(\ell^\infty, \norm{\cdot}_\infty)$. This isomorphism of topological vector spaces was first established in \citet{CiesielskiHoelder} and  extended in \cite{RackauskasBanach} to the Banach space valued case. The rest of the argument is then an elementary modification of the estimates detailed in \cite{GiraudoMO}.
\begin{proof}[Proof of Proposition \ref{StationaryTail}]
Consider the norm induced on $C_0^{\eta}$ by the Ciesielski isomorphism \cite{RackauskasBanach}, with $\lVert \cdot \rVert_\eta$ defined by $$
\lVert x\rVert_\eta \coloneqq \left\lVert x\left(1\right)\right\rVert_{\mathcal B} + \sup_{j\geq 1}2^{j\eta}\max_{1\leq k\leq 2^{j-1}}\left\lVert x\left(\left(2k\right)2^{-j}\right)-2x\left(\left(2k-1\right)2^{-j}\right) +x\left(\left(2k-2\right)2^{-j}\right)\right\rVert_{\mathcal B}.$$ As in the real-valued case,$$\sup_{0 \leq s < t \leq 1} \frac{\lVert x(t)-x(s)\rVert_{\mathcal B}}{\lvert t-s\rvert^\eta} \leq \frac14K_\eta\lVert{x}\rVert_\eta$$ with $K_\eta = \frac{4}{(2^\eta-1)(2^{1-\eta}-1)}.$ Here, the coefficient $\frac14$ is chosen with view to the estimates that follow.
Applying the equivalence of norms to our problem, we see that $$\mathbb P\left( \sup_{0 \leq s < t < 1} \frac{\norm{ X^H(t)-X^H(s)}_\mathcal{B}}{\lvert t-s\rvert^\eta}  > b\right) \leq \mathbb P\left( K_\eta \lVert X^H \rVert_\eta  > b\right)$$ and therefore, we only need to find estimates on $$
\begin{aligned}
\mathbb P\left(\lVert X^H \rVert_\eta > 4 b\right) &\leq \sum_{j=1}^{\infty}\sum_{k=1}^{2^{j-1}}
\mathbb P\left(\left\lVert X^H\left(\left(2k\right)2^{-j}\right)-2X^H\left(\left(2k-1\right)2^{-j}\right) +X^H\left(\left(2k-2\right)2^{-j}\right)\right\rVert_\mathcal{B}> 2b 2^{-j\eta}\right) \\ 
& \quad + \mathbb P\left(\left \lVert X^H(1) \right \rVert > 2b \right)\\
&\leq \sum_{j=1}^{\infty}\sum_{k=1}^{2^{j-1}}
\mathbb P\left(\left\lVert X^H\left(\left(2k\right)2^{-j}\right)-X^H\left(\left(2k-1\right)2^{-j}\right)\right \rVert_\mathcal{B} > b2^{-j\eta} \right) \\ & \quad \quad \quad \quad + \mathbb P\left(\left \lVert X^H\left(\left(2k-1\right)2^{-j}\right) - X^H\left(\left(2k-2\right)2^{-j}\right)\right \rVert_\mathcal{B}> b2^{-j\eta} \right) \\ 
& \quad + \mathbb P\left(\left \lVert X^H(1) \right \rVert_\mathcal{B} > 2b \right)
\end{aligned}
$$
For simplicity, let $Z \coloneqq X^H(1)$. By stationarity of increments and self-similarity, we observe that tail estimates of increments are equal to tail estimates of $Z$, i.e.
$$\mathbb P\left(\left\lVert X^H\left(\left(2k\right)2^{-j}\right)-X^H\left(\left(2k-1\right)2^{-j}\right)\right \rVert_\mathcal{B} > b2^{-j\eta} \right) \leq \mathbb P\left(\lVert Z\rVert_\mathcal{B} > b2^{j(H - \eta)}\right)$$
and hence 
$$
\mathbb P\left(\lVert X^H\rVert_\eta > 4b\right)\leq \sum_{j=1}^{\infty} 2^{j+1}
\mathbb P\left(\lVert Z\rVert_\mathcal{B} > b2^{j(H - \eta)}\right) + \mathbb P(\norm{Z}_\mathcal{B} > 2b).
$$
In case \eqref{PolDecay}, the claimed estimate now follows by plugging in the assumed tail estimate for $ 0 < \eta < H-1/\alpha$. In the second case, the same method gives us the estimate $$\mathbb P\left(\lVert X^H\rVert_\eta > 4b\right) \leq \mathcal O\left(\exp(-b^\alpha/k)\right)$$ by an elementary but slightly more involved estimate. Plugging in $\frac{b}{K_\eta}$ then yields the desired expression for $0 < \eta < H$.
\end{proof}

\section*{Acknowledgement}
AS is supported by Deutsche Forschungsgemeinschaft (DFG, German Research
Foundation) under Germany's Excellence Strategy - The Berlin Mathematics
Research Center MATH+ (EXC-2046/1, project ID: 390685689). WS acknowledges seed support for DFG CRC/TRR 388 “Rough Analysis, Stochastic Dynamics and Related Topics”.

\bibliography{main}
\end{document}